\documentclass[reqno,oneside]{amsart}
\usepackage[margin=3cm]{geometry}
\usepackage{amssymb,amsmath}
\usepackage{amsfonts}
\usepackage{rotating}
\usepackage{graphicx}
\usepackage[boxed,oldcommands,norelsize]{algorithm2e}
\usepackage[numbers,sort&compress]{natbib}
\usepackage{multirow}
\usepackage{hhline}
\usepackage{colortbl}
\newcolumntype{L}{>{\centering\arraybackslash}m{0.46\textwidth}}
\usepackage{arydshln}

\usepackage[utf8]{inputenc}
\usepackage[T1]{fontenc}
\usepackage{lmodern}
\usepackage{subfigure}
\usepackage{lipsum}
\usepackage{graphicx}
\usepackage{epstopdf}
\usepackage{algorithmic}
\usepackage{xcolor}
\ifpdf
  \DeclareGraphicsExtensions{.eps,.pdf,.png,.jpg}
\else
  \DeclareGraphicsExtensions{.eps}
\fi

\def\Eq#1{(\ref{eq:#1})}

\newtheorem{proposition}{Proposition}[section]

\def\myKeywords{Space-time parallelism, domain decomposition, BDDC, preconditioning, scalability}

\def\Eq#1{(\ref{eq:#1})}

\def\Ac{{\bar{\mathcal{A}}}}
\def\Ad{{{\mathcal{A}}}}

\def\V{{V}}
\def\W{{\mathcal{W}}}
\def\VbddcC{\widetilde{{V}}_{\rm c}}

\def\Vc{\bar{{V}}}
\def\Vd{{V}}

\def\Vl{{V}_\sbx}

\def\Vbddc{\widetilde{{V}}}
\def\VbddcF{{\widetilde{{V}}_{\rm f}}}
\def\VbddcF1{{\widetilde{{V}}_{\rm f}}}
\def\VbddcC1{{\widetilde{{V}}_{\rm c}}}

\def\uf{\boldsymbol{u}}

\def\E{\mathcal{E}}

\def\ne{K}
\def\ns{N}
\def\ie{k}
\def\is{n}

\def\nes{{\ne_\is}}

\def\sbxt{{\boldsymbol{\omega}_\is}}
\def\sbx{\omega}

\def\Vlxt{\boldsymbol{V}_{\sbxt}}
\def\Vxt{\boldsymbol{V}}
\def\Vcxt{{ \bar{\boldsymbol{V}}}}

\def\Vlx{{{V}_{\sbx}}}
\def\Vx{{V}}
\def\uxt{{\boldsymbol{u}}}
\def\ulxt{{\uxt_\sbxt}}
\def\lxt{{\boldsymbol{l}}}
\def\llxt{{\lxt_\sbxt}}

\def\vxt{{\boldsymbol{v}}}
\def\ulxt{{\uxt_\sbxt}}
\def\vlxt{{\vxt_\sbxt}}
\def\ulx{{u_\sbx}}
\def\vlx{{v_\sbx}}
\def\elt{{\delta_\ie}}
\def\elsbt{{\delta_\is^\ie}}
\def\sbt{{\Delta_\is}}
\def\M{{\mathcal{M}}}
\def\A{{\mathcal{A}}}

\def\Ml{{\M_\sbx}}
\def\Al{{\A_\sbx}}
\def\Axt{{\boldsymbol{\A}}}
\def\Alxt{{\Axt_\sbxt}}
\def\C{{\mathcal{C}}}
\def\Cxt{{\boldsymbol{\C}}}
\def\Clxt{{\Cxt_\sbxt}}

\def\Acxt{{\bar{\boldsymbol{\A}}}}

\def\sumsb{{\sum_{\ie = 1}^\nes} }
\def\K{{\mathcal{K}}}

\def\Kl{{\K_\sbx}}
\def\Mc{{\bar{\M}}}
\def\Kc{{\bar{\K}}}
\def\Ac{{\bar{\A}}}
\def\fxt{{\boldsymbol{f}}}
\def\lxt{{\boldsymbol{l}}}

\def\Vbddcxt {{\widetilde \Vxt}}
\def\Wx{{\mathcal{W}}}
\def\Wlx{{\Wx_\sbx}}
\def\Wlxt{{\boldsymbol{\Wx}_{\sbxt}}}
\def\Wxt{{\boldsymbol{\Wx}}}
\def\Ex{{\mathcal{E}}}
\def\Ext{{\boldsymbol{\Ex}}}
\def\Qx{{\mathcal{Q}}}
\def\Qxt{{\boldsymbol{\Qx}}}
\def\Bx{{\mathcal{B}}}
\def\Bxt{{\boldsymbol{\Bx}}}
\def\rxt{{\boldsymbol{r}}}
\def\sxt{{\boldsymbol{s}}}
\def\Vd{{V}}
\def\Vc{\bar{{V}}}
\def\Vd{{V}}

\def\Q{{\mathcal{Q}}}
\def\B{{\mathcal{B}}}
\def\Phixt{{\boldsymbol{\Phi}}}
\def\Philxt{{\boldsymbol{\Phi}_\sbxt}}
\def\Psixt{{\boldsymbol{\Psi}}}
\def\Psilxt{{\boldsymbol{\Psi}_\sbxt}}

\title{{\TheTitle}\thanks{This work has been funded by the European Research Council under the FP7 Programme Ideas through the Starting Grant No. 258443 - COMFUS: Computational Methods for Fusion Technology and the FP7 NUMEXAS project under grant agreement 611636. S. Badia gratefully acknowledges the support received from the Catalan Government through the ICREA Acad\`emia Research Program. M. Olm also acknowledges the support from the Ag\`{e}ncia de Gesti\'{o} i d'Ajuts Universitaris i de Recerca, under the FI-AGAUR 2015 grant. We acknowledge the North-German Supercomputing Alliance (HLRN) and the Barcelona Supercomputing Center (BSC) for providing computing time.}}

\usepackage{amsopn}
\DeclareMathOperator{\diag}{diag}

\title[]{Space-time balancing domain decomposition }

\author{ Santiago Badia $^{\dag\ddag}$  \and Marc Olm $^{\dag\ddag}$ }
\thanks{
	$^\dag$ Centre Internacional de M\`etodes Num\`erics en Enginyeria (CIMNE)
	\@, Parc
	Mediterrani de la Tecnologia, UPC, Esteve Terradas 5, 08860 Castelldefels,
	Spain (\{sbadia,molm\}@cimne.upc.edu). \\ 
	\indent$^\ddag$
	Universitat Polit\`ecnica de Catalunya, Jordi Girona 1-3, Edifici C1, 08034
	Barcelona, Spain.
}

\begin{document}

\begin{abstract}
In this work, we propose two-level space-time domain decomposition preconditioners for parabolic problems discretized using finite elements. They are motivated as an extension to space-time of balancing domain decomposition by constraints preconditioners. The key ingredients to be defined are the sub-assembled space and operator, the coarse degrees of freedom (DOFs) in which we want to enforce continuity among subdomains at the preconditioner level, and the transfer operator from the sub-assembled to the original finite element space. With regard to the sub-assembled operator, a perturbation of the time derivative is needed to end up with a well-posed preconditioner. The set of coarse DOFs includes the time average (at the space-time subdomain) of classical space constraints plus new constraints between consecutive subdomains in time. Numerical experiments show that the proposed schemes are weakly scalable in time, i.e., we can efficiently exploit increasing computational resources to solve more time steps in the same {total elapsed} time. Further, the scheme is also weakly space-time scalable, since it leads to asymptotically constant iterations when solving larger problems both in space and time. Excellent {wall clock} time weak scalability is achieved for space-time parallel solvers on some thousands of cores.
\end{abstract}

\maketitle


\tableofcontents

\pagestyle{myheadings}
\thispagestyle{plain}

\section{Introduction} \label{intro}

At the beginning of the next decade supercomputers are expected to reach a peak performance of one exaflop/s, which implies a 100 times improvement with respect to current supercomputers. This improvement will not be based on faster processors, but on a much larger number of processors (in a broad sense). This situation will certainly have an impact in large scale computational science and engineering. Parallel algorithms will be required to exhibit much higher levels of concurrency, keeping good scalability properties. 

In mesh-based implicit simulations, e.g., finite element (FE), finite volume, or finite difference methods, one ends up with a linear system to be solved. The linear system solve is a bottleneck of the simulation pipeline and weakly scalable algorithms require complex mathematical approaches, like algebraic multigrid (AMG) or multilevel domain decomposition (DD) techniques. When dealing with transient problems, since information always moves forward in time, one can exploit sequentiality. At every time step one has to solve a spatial problem before proceeding to the next step and parallelism can be exploited at the linear system solve. Although parallel-in-time methods are becoming popular, the sequential-in-time approach is the standard procedure in scientific simulations. However, the tremendous amounts of parallelism to be exploited in the near future certainly motivates to change this paradigm, since further concurrency opportunities must be sought.

In transient simulations, a natural way to go is to exploit concurrency not only in space but also in time. The idea is to develop space-time solvers that do not exploit sequentiality (at least at the global level) and thus provide the solution in space at all time values in one shot. Space-time parallelism is a topic that is receiving rapidly increasing attention. Different iterative methods have been considered so far. One approach is to use the Parareal method \cite{parareal}, which is a time-only parallel algorithm, combined with a parallel space preconditioner (see, e.g., \cite{gander_parareal_2013}). Another space-time algorithm is PFASST \cite{pfasst,minion_interweaving_2015}, which combines a spectral deferred correction time integration with a nonlinear multigrid spatial solver; the viability of the PFASST method has been proved in \cite{speck_space-time_2013} at JUQUEEN. Weakly scalable space-time AMG methods can be found in \cite{gander_analysis_2014,weinzierl_geometric_2012,falgout_parallel_2014}. 

The multilevel balancing DD by constraints (BDDC) preconditioner \cite{mandel_multispace_2008,tu_three-level_2007} has recently been proved to be an excellent candidate for extreme scale simulations in \cite{badia_multilevel_2016}, where a recursive implementation that permits overlapping among communication and computation at all levels has scaled up to almost half a million cores and two million subdomains (MPI tasks), for both structured and unstructured meshes with tens of billions of elements. The key ingredient of these methods relies on the definition of a FE space with relaxed inter-element continuity. These spaces are defined by choosing the quantities to be continuous among processors, i.e., the coarse degrees of freedom (DOFs) \cite{dohrmann_preconditioner_2003}. As far as we know, scalable DD methods in space-time have not been considered so far. 

In this work, we develop \emph{weakly scalable space-time preconditioners based on BDDC methods}. In order to do that, we extend the key ingredients in the space-parallel BDDC framework, namely the sub-assembled space and operator, coarse DOFs, and transfer operators, to space-time. We prove that the resulting method only involves a set of well-posed problems, and time causality can still be exploited at the local level. We have solved a set of linear and nonlinear problems that show the excellent weak scalability of the proposed preconditioners. 

The outline of the article is as follows. In Sect. \ref{sec:problem_setting} we set the problem and introduce notation. In Sect. \ref{sec:BDDC} we introduce the classical space-parallel BDDC preconditioners. In Sect. \ref{sec:STBDDC} we develop space-time BDDC (STBDDC) preconditioners. In Sect. \ref{sec:num_exp} we present a detailed set of numerical experiments showing the scalability properties of the proposed methods. Finally, in Sect. \ref{sec:conclusions} we draw some conclusions.

\section{Problem setting}\label{sec:problem_setting}

In this section, we introduce the problem to be solved, the partition of the domain in space and time, and the  space and time discretization. In the sequel, calligraphic letters are used for operators. $\mathcal{M}$ denotes mass matrix operators related to the time derivative discretization, $\mathcal{K}$ is used for the rest of terms in the PDE operator, and $\mathcal{A}$ is used for the sum of these two operators. Given an operator $\mathcal{A}$, we will use the notation $\mathcal{A}(u,v) \doteq \langle \mathcal{A} u , v \rangle $. Uppercase letters $(V,\ldots)$ are used for (FE-type) functional spaces whereas functions are represented by lowercase letters $(v,\ldots)$. We use classical functional analysis notation for Sobolev spaces. 

\subsection{Domain partitions}
We consider a bounded space-time domain $\Omega \times (0,T]$, where $\Omega$ is an open polyhedral domain $\Omega \subset \mathbb{R}^d$, $d$ being the space dimension. Let us construct two partitions of $\Omega$, a \emph{fine} partition into \emph{elements} and a \emph{coarse} partition into \emph{subdomains}. The partition of $\Omega$ into elements is represented by $\theta$. In space, elements $e \in \theta$ are tetrahedra/hexahedra for $d=3$ or triangles/quadrilaterals for $d=2$. The \emph{coarse} partition $\Theta$ of $\Omega$ into subdomains is obtained by \emph{aggregation of elements} in $\theta$, i.e., there is an element partition $\theta_\omega \doteq \{ e \in \theta : e \subset \omega \} \subset \theta$ for any $\omega\in\Theta$.  The interface of the subdomain partition is $\Gamma \doteq \cup_{\omega \in \Theta} \partial \omega \setminus \partial \Omega$. 

For the time interval $(0,T]$, we define a time partition $\{0 = t^0 , t^1, \ldots, t^\ne = T\}$ into $\ne$ time elements. We denote the $\ie$-th element by  $\elt \doteq (t_{\ie-1},t_{\ie}]$, for $\ie = 1, \ldots, \ne$.

\subsection{Space-time discretization}

Let us consider as a model problem the following transient convection-diffusion-reaction (CDR) equation: find $u \in H^1(\Omega)$ such that
\begin{align}\label{eq:cd}
&\partial_t u - \nabla \cdot \nu \nabla u + \beta \cdot \nabla u + \sigma u = f \quad \hbox{on } \, \Omega, \quad \hbox{almost everywhere in } (0,T], \\
&u = g   \quad \hbox{in } \, \partial \Omega, \qquad  u(0, x)= u^0,      \nonumber
\end{align} 
where $\nu$ and $\sigma$ are positive constants, $\beta \in \mathbb{R}^d$, and $f \in H^{-1}(\Omega)$.  We supplement this system with the initial condition $u(0) = u^0$. Homogeneous Dirichlet data is assumed for the sake of clarity in the exposition, but its extension to the general case is obvious. {Besides, let us consider  $\beta=0$ and $\sigma=0$ for simplicity in the exposition of the algorithm.}  In Sect. \ref{sec:num_exp} we will take a nonlinear viscosity $\nu(u) = \nu_0 | \nabla u|^p$ (with $p \geq 0$ and $\nu_0 > 0$), i.e., the transient $p$-Laplacian problem.

For the space discretization, we use $H^1$-conforming FE  spaces on conforming meshes with strong imposition of Dirichlet conditions. The discontinuous Galerkin (DG) case will not be considered in this work, but we refer to \cite{dryja_bddc_2007} for the definition of BDDC methods for DG discretizations. 
We define $\Vc \subset H_0^1(\Omega)$ as the global FE space related to the FE mesh $\theta$. Further, we define the FE-wise operators:
\begin{align*}
\M_e(u,v) & \doteq \int_e u v , \quad \K_e(u,v) \doteq \int_e \nu \nabla u \cdot \nabla v , \\  \A_e(u,v) & \doteq \M_e(u,v) + | \elt | \K_e(u,v). 
\end{align*}
The global FE problem $\Ac : \Vc \rightarrow \Vc$ can be written as the sum of element contributions, i.e.,
$$
\Ac(u,v) \doteq \sum_{e \in \theta} \A_e(u,v), \qquad \hbox{for } u, v \in \Vc. \qquad \hbox{(Analogously for  $\Mc$ and $\Kc$.)}
$$
In time, we make use of a collocation-type method of lines. For the sake of clarity, we will use the Backward-Euler time integration scheme. In any case, the resulting preconditioner can readily be applied to any $\theta$-method or Runge-Kutta method. We are interested in solving the following fully discrete system: given $u(t^0) = 0$, find at every time step $\ie = 1, \ldots, \ne$ the solution $u(t^\ie) \in \Vc$ of
\begin{equation}\label{eq:original_problem}
 \Ac u(t^\ie) =\Mc u(t^\ie)  + | \elt | \Kc u(t^\ie)  = \bar{g}(t^\ie), \qquad \hbox{ for any } v \in \Vc,
\end{equation}
with $\bar{g}(t^\ie) \doteq| \elt |  \bar{f}(t^\ie) + \Mc u(t^{\ie-1})  \in \Vc'$, where  $\Vc'$ denotes the dual space of $\Vc$. 
Non-homogeneous boundary conditions can be enforced by simply modifying the right-hand side at $t^1$, i.e., $\bar{g}(t^1) \doteq| \elt |  \bar{f}(t^1) + \Mc u^0  \in \Vc'$. Such imposition of boundary conditions, i.e., by enforcing homogeneous conditions in the FE space plus the modification of the right-hand side, is better suited for the space-time framework. (We note that this is the way strong Dirichlet boundary conditions are imposed in FE codes.)

\section{Space BDDC preconditioning}\label{sec:BDDC}

In this section, we present first a parallel solver for the transient problem \Eq{original_problem}, which combines a sequential-in-time approach with a space-parallel BDDC preconditioned Krylov solver at every time step \cite{dohrmann_preconditioner_2003}. It will serve to introduce space-parallel BDDC methods and related concepts that will be required in the space-time section. BDDC preconditioners involve the definition of three key ingredients: (1) a sub-assembled problem that involves independent subdomain corrections, (2) a set of coarse DOFs and the corresponding subspace of functions with continuous coarse DOFs among subdomains, and (3) the interior correction and transfer operators. Let us elaborate these ingredients.

\subsection{Sub-assembled problem}

Non-overlapping DD preconditioners rely on the definition of a sub-assembled FE problem, in which contributions between subdomains have not been assembled. In order to do so, at every subdomain $\omega \in \Theta$, we consider the FE space $\Vd_\omega$ associated to the element partition $\theta_\omega$ with homogeneous boundary conditions on $\partial\omega \cap \partial \Omega$. One can define the subdomain operator $
\Al(u,v) = \sum_{e \in \theta_\omega} \A_e(u,v)$, for $u, v \in \Vl$. (Analogously for  $\Ml$ and $\Kl$.)

Subdomain spaces lead to the sub-assembled space of functions $\V \doteq \Pi_{\omega \in \Theta} \V_\omega$. For any $u \in \Vd$, {we define its restriction to a subdomain $\omega \in \Theta$ as $u_\omega$}. Any function $u \in \Vd$ can be represented by  its unique decomposition into subdomain functions as $\{u_\omega \in \Vlx\}_{\omega \in \Theta}$. 
 We also define the sub-assembled operator $\mathcal{A}(u,v) \doteq \Pi_{\omega \in \Theta} {\mathcal{A}}_\omega (u_\sbx,v_\sbx)$. (Analogously for $\M$ and $\K$.) 

With these definitions,  $\Vc$ can be understood as the subspace of functions in $\Vd$ that are continuous on the interface $\Gamma$, and $\bar{\mathcal{A}}$ can be interpreted as the Galerkin projection of $\mathcal{A}$ onto $\Vc$. We note that $\theta$ and the FE type defines $\Vc$, whereas $\Theta$ is also required to define the local  spaces $\{\V_\omega\}_{\omega \in \Theta}$ and the sub-assembled space $\Vd$, respectively.

At this point, we can state the following sub-assembled problem: given $g \in \Vx'$, find $u \in \Vx$ such that 
\begin{align}\label{eq:sub-assembled_space}
& \Al \ulx = \Ml \ulx +  | \elt | \Kl \ulx  = g_\sbx, \quad \sbx \in \Theta.
\end{align}
With the previous notation, we can write the sub-assembled problem in a compact manner as $\Ad u = g$. 

\subsection{Coarse DOFs}

A key ingredient in DD preconditioners is to classify the set of nodes of the FE space $\Vc$. The interface $\partial e$ of every FE in the mesh $\theta$ can be decomposed into vertices, edges, and faces. By a simple classification of these entities, based on the set of subdomains that contain them, one can also split the interface $\Gamma$ into faces, edges, and vertices (at the subdomain level), that will be called geometrical objects. We represent the set of geometrical objects by $\Lambda$. In all cases, edges and faces are open sets in their corresponding dimension. By construction, faces belong to two subdomains and edges belong to more than two subdomains in three-dimensional problems. This classification of $\Omega$ into objects automatically leads to a partition of interface DOFs into DOF objects, due to the fact that every DOF in a FE does belong to only one geometrical entity.  
These definitions are heavily used in DD preconditioners (see, e.g., \cite[p.~88]{Toselli2005}). 

Next, we associate to some (or all) of these geometrical objects a \emph{coarse} DOF. In BDDC methods, we usually take as coarse DOFs mean values on a subset of objects $\Lambda_O$. Typical choices of $\Lambda_O$ are $\Lambda_O \doteq \Lambda_{C}$, when only corners are considered, $\Lambda_O \doteq \Lambda_C \cup \Lambda_{E}$, when corners and edges are considered, or $\Lambda_O \doteq \Lambda$, when corners, edges, and faces are considered. These choices lead to three common variants of the BDDC method referred as BDDC(c), BDDC(ce) and BDDC(cef), respectively. This classification of DOFs into objects can be restricted to any subdomain $\omega \in \Theta$, leading to the set of subdomain objects $\Lambda_O(\omega)$.

With the classification of the interface nodes and the choice of the objects in $\Lambda_O$, we can define the coarse DOFs and the corresponding BDDC space. Given an object $\lambda \in \Lambda_O(\omega)$, let us define its restriction operator $\tau_\lambda^\omega$ on a function $u \in \Vl$ as follows: $\tau_\lambda^\omega(u)(\xi) = u(\xi)$ {for  } a node $\xi$ that belongs to the geometrical object $\lambda$, and zero otherwise. We define the BDDC space $\Vbddc \subset \Vx$ as the subspace of functions $v \in \Vx$ such that the constraint
\begin{equation}\label{eq:coarse-dofs}
\int_{\lambda} \tau_\lambda^\sbx(\vlx)  \qquad \hbox{is identical for all } \sbx \in {\rm neigh}(\lambda){,}  
\end{equation}
{where ${\rm neigh}(\lambda)$ stands for the set of subdomains that contain the object $\lambda$.} (The integral on $\lambda$ is just the value at the vertex, when $\lambda$ is a vertex.) Thus, every $\lambda \in \Lambda_O$ defines a coarse DOF value \Eq{coarse-dofs} that is continuous among subdomains. Further, we can define the BDDC operator $\widetilde \Ad$ as the Galerkin projection of $\Ad$ onto $\Vbddc$. 

\subsection{Transfer operator}

The next step is to define a transfer operator from the sub-assembled space $\Vd$ to the continuous space $\Vc$. The transfer operator is the composition of a weighting operator and a harmonic extension operator.

\begin{enumerate}

\item The weighting operator $\mathcal{W}$ takes a function $u \in \Vd$ and computes mean values on interface nodes, i.e.,
\begin{equation}\label{eq:weighting}
\mathcal{W}u(\xi) \doteq \frac{\sum_{\omega \in {\rm neigh}(\xi)}u_\omega(\xi)}{|{\rm neigh}(\xi)|}, 
\end{equation} at every node $\xi$ of the FE mesh $\theta$, where ${\rm neigh}(\xi)$ stands for the set of subdomains that contain the node $\xi$. It leads to a continuous function $\mathcal{W}u \in \Vc$. It is clear that this operator only modifies the DOFs on the interface. Other choices can be defined for non-constant physical coefficients. 

\item Next, let us define the bubble space $\V_0 \doteq \{ v \in \V: v = 0 \hbox{ on } \Gamma \}$ 
and the Galerkin projection $\Ad_0$ of $\Ad$ onto $\Vx_0$. {We also define the trivial injection $\mathcal{I}_0$ from $\V_0$ to $\Vc$. The \emph{harmonic} extension reads as $\mathcal{E} v \doteq (I - \mathcal{I}_0 \Ad_0^{-1} \mathcal{I}_0^T \Ac) v$.} The computation of $\Ad_0^{-1}$ involves to solve problem \Eq{sub-assembled_space} with homogeneous boundary conditions on $\Gamma$.
This operator corrects interior DOFs only.
\end{enumerate}
The transfer operator $\Q : \Vd \rightarrow \Vc$ is defined as $\Q \doteq \E \Wx$. 

\subsection{Space-parallel preconditioner}

With all these ingredients, we are now in position to define the BDDC preconditioner. This preconditioner is an additive Schwarz preconditioner (see, e.g. \cite[chap.~2]{Toselli2005}), with corrections in $\V_0$ and the BDDC correction in $\Vbddc$ with the transfer $\Q$. As a result, the BDDC preconditioner reads as:
{
  \begin{equation}\label{eq:bddc}
\B \doteq \mathcal{I}_0 \A_0^{-1} \mathcal{I}_0^T+ \Q \widetilde \A^{-1} \Q^T.
  \end{equation}
  }

\section{Space-time BDDC preconditioning}\label{sec:STBDDC}
As commented in Sect. \ref{intro}, the huge amounts of parallelism of future supercomputers will require to seek for additional concurrency. In the simulation of \Eq{original_problem} using the space-parallel preconditioner \Eq{bddc}, we are using a sequential-in-time approach by exploiting time causality. The objective of this section is to solve \Eq{original_problem} at all time steps in one shot, by using a space-time-parallel preconditioner and a Krylov subspace method for non-symmetric problems. In order to do so, we want to extend the BDDC framework to space-time. 

In the following, we will use bold symbols, e.g., $\boldsymbol{u}$, $\boldsymbol{\V}$, or $\boldsymbol{\mathcal{A}}$, for space-time functions, functional spaces, and operators, respectively. $\boldsymbol{I}$ is the identity matrix, which can have different dimension in different appearances.

First, we must start with a space-time partition of $\Omega \times (0,T]$. We consider a time subdomain partition by aggregation of time elements, $\{0 = T_0, T_1, \ldots, T_\ns = T \}$ into $\ns$ time subdomains.  We denote the $\is$-th subdomain as  $\sbt \doteq (T_{\is-1},T_{\is}]$, for $\is = 1, \ldots, \ns$. By definition, $\sbt$ admits a partition into $\ne_\is$ time elements $\{T_{\is-1}= t^0_\is, \ldots, t^{\ne_\is}_\is =T_\is \}$. Next, we define the space-time subdomain partition as the Cartesian product of the space subdomain partition $\Theta$ and the time subdomain partition defined above; for every space subdomain $\omega$ and time subdomain $\sbt$, we have the space-time subdomain $\sbxt \doteq \sbx \times \sbt$. 

 The global space of continuous space-time functions in which we want to solve \Eq{original_problem} is {the FE space $\Vc$ of spatial functions times $K+1$ time steps, i.e., } $\Vcxt \doteq [\Vc]^{\ne+1}$, constrained to zero initial condition. Thus, by definition, $\uxt \in \Vcxt$ can be expressed as $\uxt = ( u(t^0) = 0, u(t^1), \ldots,  u(t^\ne))$, and the original problem \Eq{original_problem} (for all time step values) can be stated in a compact manner as 
\begin{equation}\label{eq:op_compact}
\Acxt \uxt = \bar{\boldsymbol{f}}, \qquad \hbox{in } \Vcxt.
\end{equation}
In order to define the STBDDC preconditioner for \Eq{op_compact}, we will use the same structure as above, extending the three ingredients in Sect. \ref{sec:BDDC} to the space-time case.

\subsection{Sub-assembled problem}

Using the space-time partition above, the trial (and test) space for the local space-time subdomain $\sbxt$ is $\Vlxt \doteq [ \Vlx ]^{\nes+1}$. Thus, by definition, $\ulxt \in \Vlxt$ can be expressed as $\ulxt = (\ulx(t_\is^0), \ldots, \ulx(t_\is^\nes))$. Analogously, the sub-assembled space is $\Vxt \doteq \Pi_{\is = 1}^\ns \Pi_{\omega \in \Theta} \Vlxt$. Let us note that, using these definitions, functions in $\Vxt$ have duplicated values at $T_\is$, for $\is = 1, \ldots, \ns-1$. The global space of continuous space-time functions  $\Vcxt$ can be understood as the subspace of functions in $\Vxt$ that are continuous on the space-time interface $\Gamma \times T_\is$, for $\is = 1, \ldots, \ns-1$. 

Now, we propose the following sub-assembled problem in $\Vxt$: find the solution $\uxt \in \Vxt$ such that, at every $\sbxt$ in the space-time partition, it satisfies the space-time problem
\begin{align}\label{eq:space-time-problem}
&\sumsb  \left\{\Ml(\ulx(t_\is^\ie) - \ulx(t_\is^{\ie-1}), \vlx(t_\is^\ie) ) +  | {\elsbt} | \Kl (\ulx(t_\is^\ie), \vlx(t_\is^\ie)) \right\} \\
 & + \frac{{(1 - \rm Kr}_{1,\is})}{2} \Ml( \ulx(t_\is^0), \vlx(t_\is^0) )  - \frac{{(1 - \rm Kr}_{\ns,\is})}{2} \Ml ( \ulx(t_\is^\nes) , \vlx(t_\is^\nes) )  \nonumber \\ & = \sumsb  | {\elsbt} | f_\sbx ( t_\is^\ie ) (\vlx(t_\is^\nes)) , \nonumber
\end{align}
for any $\vxt \in \Vxt$, where ${\rm Kr}_{i,j}$ is the Kronecker delta. {Note that the perturbation terms in the second line of \Eq{space-time-problem} are introduced only on time interfaces, i.e., in the first and last time steps of the time subinterval, as long as the corresponding time step is not a time domain boundary. For subdomains with $n=1$, and thus $t_\is^0 = 0$, the first stabilization term vanishes. Analogously, the second stabilization term vanishes for $n = N$ and $t_\is^\nes = T$.} We can write \Eq{space-time-problem} in compact manner as 
\begin{equation}\label{eq:sp_compact}
\Axt \uxt = \fxt \qquad \hbox{in } \Vxt.
\end{equation} 
The motivation of the perturbation terms is to have a positive semi-definite sub-assembled operator. In any case, the perturbation terms are such that, after inter-subdomain assembly, we recover the original space, i.e., $\Axt$ is in fact a sub-assembled operator with respect to $\Acxt$, {since interface perturbations between subdomains cancel out.} We prove that these properties hold.
\begin{proposition}\label{perturbation}
The Galerkin projection of the sub-assembled space-time problem \Eq{sp_compact} onto $\Vcxt$ reduces to the original problem \Eq{op_compact}. Further, the sub-assembled operator $\Axt$ is positive definite.
\end{proposition}
\begin{proof}
In order to prove the equivalence, we need to show that $\Acxt$ is the Galerkin projection of $\Axt$ onto {$\Vcxt$}, which amounts to say that the perturbation terms vanish for $\uxt \in {\Vcxt}$. First, we note that the following equality
$$
\sum_{\is = 1}^\ns \left(\frac{{(1-\rm Kr}_{1,\is})}{2}  \ulx(t_\is^0) - \frac{{(1-\rm Kr}_{\ns,\is})}{2}  \ulx(t_\is^\nes) \right) = 0
$$
holds for functions that are continuous in time, {since $\ulx(t_{n-1}^{K_{n-1}}) = \ulx(t_{n}^{0})$ for $n=2,\ldots,N$}. On the other hand, multiplying the right-hand of \Eq{space-time-problem} against $\vxt = \uxt$, using the fact that $(a-b)a = \frac{1}{2}(a^2 - b^2) + \frac{1}{2}(a-b)^2$, we get 
\begin{equation}\label{eq:energy}
\Axt(\uxt,\uxt) = \frac{1}{2} \| u (T) \|^2 + \sum_{\is = 1}^\ns  \sum_{\ie=1}^\ne \left( \frac{1}{2} \| u(t_\is^\ie) - u(t_\is^{\ie-1}) \|^2 + | {\elsbt} | \K(  u(t_\is^\ie), u(t_\is^\ie)) \right).
\end{equation}
{Since $\K$ is positive semi-definite, $\Axt$ is positive semi-definite. On the other hand, $\Axt$ is a lower block triangular matrix. Restricted to one subdomain block, it has diagonal blocks $\frac{1}{2}\Ml$ at the first time step, $\Ml + | {\delta_\ie} | \Kl$ at intermediate time steps, and $\frac{1}{2}\Ml + | {\delta_\ie} | \Kl$ at the last time step. Since all these matrices are invertible, $\Axt$ is non-singular. Further, $\Axt^T$ is an upper triangular non-singular matrix. As a result, $\Axt$ is positive-definite.}
\end{proof}


\subsection{Coarse DOFs}

Let us define the continuity to be enforced among space-time subdomains. Let us consider a set of space objects $\Lambda_O$ (see Sect. \ref{sec:BDDC}). We define $\Vbddcxt \subset \Vxt$ as the subspace of functions $\vxt \in \Vxt$ such that the constraint
\begin{equation}\label{eq:constraint_space_xt}
{\sum_{\ie =1}^{\nes-1}}  | {\elsbt} | \int_{\lambda} \tau_\lambda^\sbx(\vlx(t_\is^\ie)) \qquad \hbox{is identical for all } \sbx \in {\rm neigh}(\lambda), 
\end{equation}
holds for every $\lambda \in \Lambda_O$, and 
\begin{equation}\label{eq:constraint_time_xt}
\int_\sbx \vlx(t_\is^0) = 
\int_\sbx \vlx(t_{\is-1}^{\ne_{\is-1}}), \qquad \hbox{for all } \, \sbx \in \Theta, \quad \is = 2, \ldots, \ns.
\end{equation}
The first set of constraints are the mean value of the space constraints in \Eq{coarse-dofs} over time sub-intervals $\Delta_\is$. The second constraint enforces continuity between two consecutive-in-time subdomains $\boldsymbol{\sbx}_{\is-1}$ and $\sbxt$ of the mean value of the function on their corresponding space subdomain $\sbx$. The Galerkin projection of $\Axt$ onto $\Vbddcxt$ is denoted by $\widetilde{\Axt}$.

Additionally, motivated by a space-time definition of objects, i.e., applying the object generation above to space-time meshes, we could also enforce the continuity of the coarse DOFs
\begin{align}\label{eq:coarse-dofs2}
\int_{\lambda} \tau_\lambda^\sbx(\vlx(t_\is^0)), \ \is = 2, \ldots, \ns \ \ \hbox{and} \ \
\int_{\lambda} \tau_\lambda^\sbx(\vlx(t_\is^\nes)), \ \is = 1, \ldots, \ns-1,
\end{align}
for every $\lambda \in \Lambda_O$. Thus, we are enforcing pointwise in time (in comparison to the mean values in \Eq{constraint_space_xt}) space constraints on time interfaces. Figure \ref{fig:objects} illustrates the resulting space-time set of objects where continuity is to be enforced in a sub-assembled space $\Vxt$. 

\begin{figure}[h!!]
\begin{center}
\includegraphics[width=0.57\textwidth]{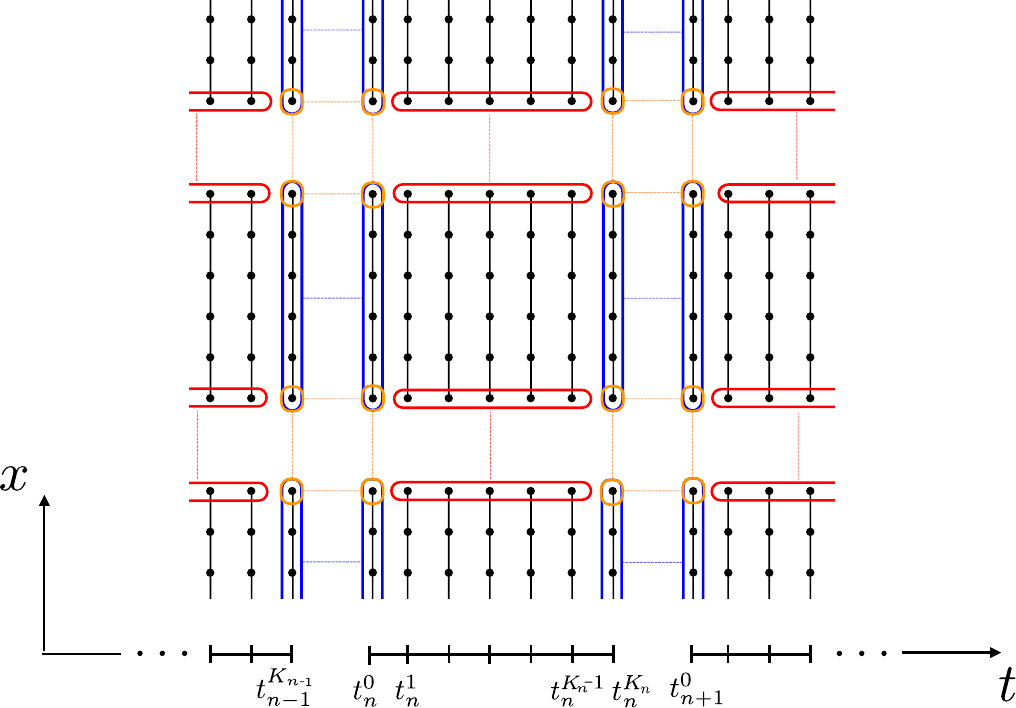}
\caption{Continuity to be enforced among space-time subdomains, for the 1-dimensional spatial domain case. The sets of nodes in red are related to the {\color{red} space constraints time averages over time sub-intervals} in \Eq{constraint_space_xt}, the ones in blue are the {\color{blue} space mean value constraints on time interfaces} in \Eq{constraint_time_xt}, and the ones in orange  are {\color{orange} spatial} {\color{orange} constraints on time interfaces} in \Eq{coarse-dofs2}.}
\end{center}
\label{fig:objects}
\end{figure}

\subsection{Transfer operator}
Next, we have to define a transfer operator from $\Vxt$ to $\Vcxt$, and the concept of harmonic extension in the space-time setting.
\begin{enumerate}
\item In order to define the space-time weighting operator, we make use of the spatial-only definition in \Eq{weighting}. Let us define the subdomain restriction of the weighting operator as $\Wlx u \doteq (\W u)_\sbx$. We define the space-time weighting operator restricted to $\sbxt$ as
\begin{equation}\label{eq:weighting_xt}
\Wlxt  \uf \doteq ( \Wlx u(t_{\is-1}^{\ne_{\is-1}}), \Wlx u(t_{\is}^{1}) , \ldots, \Wlx u(t_{\is}^{\nes}) ),
\end{equation}
and $\Wxt \doteq  \Pi_{\is = 1}^\ns \Pi_{\omega \in \Theta} \Wlxt \ulxt$. We can observe that this weighting operator uses the space-only weighting operator in \Eq{weighting}, in order to make the functions continuous in space. On the other hand, on the time interfaces $T_\is$ between subdomains $\sbxt$ and $\boldsymbol{\omega}_{\is+1}$ (for $\is = 1,\ldots,\ns-1$, where functions in $\Vxt$ can be discontinuous in time) we take the value at the preceding subdomain, i.e., $\boldsymbol{\omega}_{\is}$. This choice is motivated by the causality of the problem in time.

\item Next, we define the space-time interior correction. In order to do so, we first define the space-time {``bubble''} space as $\Vxt_0$, where its local component at $\sbxt$ is
$$
\vlxt = ( 0, \vlx(t_\is^1), \ldots, \vlx(t_\is^\nes) ), \qquad v(\cdot) \in \Vx_0.
$$
This definition of $\Vxt_0$ naturally arises from the definition of the weighting operator \Eq{weighting_xt}. The nodes that are enforced to be zero in $\Vxt_0$ are the ones that are modified by \Eq{weighting_xt}.  {$\boldsymbol{\mathcal{I}}_0$ is the trivial injection from $\Vxt_0$ to $\Vcxt$ and we denote the Galerkin projection of $\Axt$ as $\Axt_0$.} Its inverse involves local subdomain problems like \Eq{space-time-problem} with zero initial condition and homogeneous boundary conditions on $\Gamma$. Finally, we define the space-time  ``harmonic'' extension operator as $\Ext \vxt \doteq (\boldsymbol{I} -\boldsymbol{\mathcal{I}}_0 \Axt_0^{-1} \boldsymbol{\mathcal{I}}_0^T \Acxt) \vxt$. {Functions $\vxt \in \Vcxt$ such that $\Ext \vxt = \vxt$ are denoted as ``harmonic'' functions.}
\end{enumerate}
We finally define the transfer operator $\Qxt: \Vxt \rightarrow \Vcxt$ as $\Qxt \doteq \Ext \Wxt$. 

\subsection{Space-time-parallel preconditioner}
After extending the previous ingredients to space-time, we can now define the STBDDC preconditioner as
{
\begin{equation}\label{eq:stbddc}
\Bxt \doteq \boldsymbol{\mathcal{I}}_0\Axt_0^{-1}\boldsymbol{\mathcal{I}}_0^T + \Qxt \widetilde \Axt^{-1} \Qxt^T.
\end{equation}
}
In the following section, we will analyse the quality of $\Bxt$ as a preconditioner for $\Acxt$. We are particularly interested in the weak scalability properties of the preconditioner. Again, this preconditioner can be cast in the additive Schwarz theory.


\subsection{Implementation aspects}
Let us make some comments about the efficient implementation of the STBDDC preconditioner. 
We want to solve system \Eq{original_problem} (or equivalently \Eq{op_compact}) for all time steps in one shot by using a Krylov iterative solver preconditioned with the STBDDC preconditioner \Eq{stbddc}. {As usual in DD preconditioning, 
it is common to take as initial guess for the Krylov solver the interior correction $\uxt^0 = \boldsymbol{\mathcal{I}}_0 \Axt_0^{-1} \boldsymbol{\mathcal{I}}_0^T \fxt$. In this case, it can be proved by induction that applying a Krylov method with $\Bxt$ as a preconditioner will give at each iterate $\Vxt_0$-orthogonal residuals of the original problem \Eq{op_compact} and ``harmonic'' directions} 
(see, e.g., \cite{Mandel2003}). 
Thus, the application of the BDDC preconditioner applied to $\rxt \in \Vxt'$ such that $\rxt \perp \Vxt_0$ can be simplified as:
$$
\Bxt \rxt = \Ext \Wxt \widetilde \Axt^{-1} \Wxt^T \rxt.
$$
It involves the following steps.

\begin{enumerate}
\item Compute $\sxt \doteq \Wxt^T \rxt$.
By the definition in \Eq{weighting_xt}, the restriction of $\sxt = \Wxt^T \rxt$ to $\sbxt$ is $\sxt_\sbxt = (0, \Wx^T_\sbx r(t_\is^1), \ldots, \Wx^T_\sbx r(t_\is^\nes) )$, where $\W_\sbx = \diag( 1/|{\rm neigh(\xi)}| )$. This operation implies nearest neighbour communications only. 

\item Compute $\widetilde \Axt^{-1} \sxt$.
In order to compute this problem, we first use the following decomposition of $\Vbddcxt$ into the subspaces $\Vbddcxt_F$ and $\Vbddcxt_C$. $\Vbddcxt_F$ is the set of functions that vanish on the coarse DOFs \Eq{constraint_space_xt}-\Eq{coarse-dofs2}. $\Vbddcxt_C$ is the complement of $\Vbddcxt_F$, which provides the values on the coarse DOFs. We define $\Vbddcxt_C$ as the span of the columns of $\Phixt$, where $\Phixt$ is the solution of 
\begin{equation}\label{eq:coarse_shape}
\left[
\begin{array}{cc}
\Alxt & \Clxt^T \\
\Clxt & \boldsymbol{0}
\end{array} \right] 
\left[
\begin{array}{c}
\Philxt \\
\llxt
\end{array} \right] 
=
\left[
\begin{array}{c}
\boldsymbol{0} \\
\boldsymbol{I}
\end{array} \right], 
\end{equation}
{where we have introduced the notation $\Clxt$ for the matrix associated to the coarse DOFs constraints}. We can check that (see \cite[p.~206]{Brenner2010} for the symmetric case) (1) $\Vbddcxt_F \perp_\Axt \Vbddcxt_C$, (2) $\Vbddcxt = \Vbddcxt_F \oplus  \Vbddcxt_C$. The local problems in \Eq{coarse_shape} are indefinite (and couple all time steps in one subdomain). In order to be able to use sequential-in-time local solvers and sparse direct methods for positive-definite matrices, we propose the following approach (see \cite{dohrmann_preconditioner_2003} for the space-parallel BDDC preconditioner). Using the fact that $\Alxt$ is non-singular (see Proposition \ref{perturbation}), we can solve \Eq{coarse_shape} using the Schur complement:
\begin{equation}\label{eq:coarse_shape_schur}
- \Clxt \Alxt^{-1} \Clxt^T \llxt = \boldsymbol{I}, \qquad \Philxt = - \Alxt^{-1} \Clxt^T \llxt.
\end{equation}
Further, for non-symmetric problems (as the space-time problem considered herein), we also require to define $\Vbddcxt_C^*$ as the span of the columns of $\Psixt$, where $\Psixt$ is the solution of 
$$
\left[
\begin{array}{cc}
\Alxt^T & \Clxt^T \\
\Clxt & \boldsymbol{0}
\end{array} \right] 
\left[
\begin{array}{c}
\Psilxt \\
\llxt
\end{array} \right] 
=
\left[
\begin{array}{c}
\boldsymbol{0} \\
\boldsymbol{I}
\end{array} \right].
$$
This problem is similar to \Eq{coarse_shape}, but replacing $\Alxt$ by $\Alxt^T$ ($\Alxt^T$ is an upper triangular non-singular matrix from Proposition \ref{perturbation}). Thus, we can use the Schur complement approach (like in \Eq{coarse_shape_schur}) to exploit sequentiality (backward in time) for the local problems. 
\end{enumerate}

With these spaces, the original problem to be solved, $\widetilde \Axt \uf = \sxt$, can be written as: find $\uxt = \uxt_F + \uxt_C \in \Vbddcxt$, where $\uxt_F \in \Vbddcxt_F$ and  $\uxt_C \in \Vbddcxt_C$ satisfy
$$
\Axt(\uxt_F , \vxt_F ) + \Axt(\uxt_C , \vxt_C^* ) = {(\boldsymbol{s},\vxt_F)} +  {(\boldsymbol{s},\vxt_C^*)} , \hbox{ for any } \vxt_F \in \Vbddcxt_F, \ \vxt_C^* \in \Vbddcxt_C^*,
$$
{where we have used the orthogonality property
$$
\Axt(\uxt_F + \uxt_C, \vxt_F + \vxt^*_C) = \Axt(\uxt_F , \vxt_F ) + \Axt(\uxt_C , \vxt_C^* ).
$$} 
Thus, it involves a fine problem and a coarse problem that are independent. The computation of the fine problem has the same structure as \Eq{coarse_shape} (with a different forcing term), and can be computed using the Schur complement approach in \Eq{coarse_shape_schur}. The Petrov-Galerkin type coarse problem couples all subdomains and is a basis for having a weakly scalable preconditioner. Its assembly, factorization, and solution is centralised in one processor or a subset of processors. 

Summarising, the STBDDC preconditioner can be implemented in such a way that standard sequential-in-time solvers can still be applied for the local problems. Due to the fact that coarse and fine problems are independent, we can exploit an overlapping implementation, in which computations at fine/coarse levels are performed in parallel. This implementation has been proved to be very effective at extreme scales for space-parallel BDDC solvers in \cite{badia_highly_2014,badia_multilevel_2016,badia_scalability_2015}. The implementation of the STBDDC preconditioner used in Sect. \ref{sec:num_exp} also exploits this overlapping strategy.

\section{Numerical experiments}\label{sec:num_exp} 

In this section we evaluate the weak scalability {for the CDR problem \Eq{cd} of the proposed STBDDC preconditioner, when combined with the right-preconditioned version of the iterative Krylov-subspace method GMRES.} The STBDDC-GMRES solver is tested with the 2D CDR PDE on regular domains. Domains are discretized with structured Q1 FE meshes and backward-Euler time integration is performed with a constant step size {$|\delta_k|$}. As performance metrics, we focus on the number of STBDDC preconditioned GMRES iterations required for convergence, and the total computation time. This time will include both preconditioner set-up and the preconditioned iterative solution of the linear system in all the experiments reported. The nonlinear case is linearized with a Picard algorithm and a relaxation factor of $\alpha=0.75$. The stopping criteria for the iterative linear solver is the reduction of the initial residual {algebraic} $\ell_2$-norm by a factor  $10^{-6}$. The nonlinear Picard algorithm stopping criteria is the reduction of the {algebraic} $\ell_2$-norm of the nonlinear residual below $10^{-3}$. 

The problem to be solved is the CDR problem \Eq{cd}. We may consider the Poisson problem for $\beta = (0,0)$ and $\sigma = 0$. Further, we will also tackle the $p$-Laplacian problem, by taking the nonlinear viscosity $\nu(u) = \nu_0 | \nabla u|^{p}$ (with $p \geq 0$ and $\nu_0 > 0$), $\beta={(0,0)}$ and $\sigma = 0$. 

\subsection{Experimental framework}

The novel techniques proposed in this paper for the STBDDC-GMRES solver have been implemented in FEMPAR. FEMPAR, developed by the Large Scale Scientific Computing (LSSC) team at CIMNE-UPC, is a parallel hybrid OpenMP/MPI, object-oriented software package for the massively parallel FE simulation of multiphysics problems governed by PDEs. Among other features, it provides the basic tools for the efficient parallel distributed-memory implementation of substructuring DD solvers~\cite{badia_implementation_2013,badia_highly_2014,badia_multilevel_2016}. The parallel codes in FEMPAR heavily use standard computational kernels provided by (highly-efficient vendor implementations of) the BLAS and LAPACK. Besides, through proper interfaces to several third party libraries, the local constrained Neumann problems and the global coarse-grid problem can be solved via sparse direct solvers. FEMPAR is released under the GNU GPL v3 license, and is more than 200K lines of Fortran95/2003/2008 code long.  Here, we use the overlapped BDDC implementation proposed in \cite{badia_highly_2014}, with excellent scalability properties. It is based on the overlapped computation of coarse and fine duties. {As long as} coarse duties can be fully overlapped with fine duties, perfect weak scalability can be attained. We refer to \cite{badia_highly_2014} and \cite{badia_multilevel_2016} for more details. Results reported in this section were obtained on two different distributed-memory platforms: the Gottfried complex of the HLRN-III Cray system, located in Hannover (Germany) and the MareNostrum III, in the Barcelona Supercomputing Centre (BSC). In all cases, we consider a one-to-one mapping among subdomains, cores and MPI tasks, plus one additional core for the coarse problem (see \cite{badia_implementation_2013,badia_highly_2014} for details).


\subsection{Weak scalability setting}

In computer science parlance weak scalability is related to how the solution time varies with the number of processors for a fixed problem size per processor. When the time remain asymptotically constant, we say that we have a scalable algorithm. When we consider problems that are obtained after discretization of differential operators, the concept of weak scalability is suitable \emph{as soon as} the relation between the different terms in the (discretization of the) PDE remains constant in the weak scalability analysis. This is the case in most scalability analyses of PDE solvers, which usually deal with steady Poisson or linear elasticity problems. However, the situation becomes more involved as one faces more complicated problems, that combine multiple differential terms of different nature. The simplest example is the CDR equation \Eq{cd}. One can consider a fixed domain $\Omega$ and fixed physical properties, and produce a weak scalability analysis by increasing the number of elements (i.e., reducing $h$) in FEs, and the number of subdomains (i.e., reducing $H$) in the same proportion.  However, as we go to larger scales, the problem to be solved tends to a simple Poisson problem (convective terms are $\mathcal{O}(1/h)$ whereas diffusive terms are $\mathcal{O}(1/h^2)$). The same situation happens for space-only parallelization of transient problems because the CFL changes in the scalability analysis. This situation has already been identified in \cite{falgout_parallel_2014,cyr_stabilization_2012}, leading to what the authors call CFL-constant scalability. In these simulations, the CFL is constant, but still, spatial differential terms can change their relative weight in the scalability analysis, e.g., one keeps the convective CFL, { i.e., $\text{CFL}_{\beta} = |\beta|\frac{|\delta|}{h}$,}  but not the diffusive CFL, {i.e., $\text{CFL}_{\nu}=\nu \frac{|\delta|}{h^2}$} (see \cite{cyr_stabilization_2012}). 

Weak scalability analysis of PDE solvers should be such that the relative weight of all the discrete differential operators is kept. To do that, we keep fixed the physical problem to be solved (boundary conditions, physical properties, etc), the FE mesh size $h$, and the subdomain size $H$, but increase (by scaling) the physical domain $\Omega \rightarrow \alpha \Omega$ and subsequently the number of subdomains and FEs. Let us consider that $\Omega = [0,1]^d$, a FE mesh of size $h = (1/n_h)^d$, and a subdomain size $H = (1/n_H)^d$. Now, we consider $\Omega' = \alpha \Omega = [0, \alpha]^d$, $\alpha \in \mathbb{N}^+$. The FE partition now must involve $\alpha n_h$ FE partitions per dimension ($\alpha^d n_h^d$ FEs) and $\alpha n_H$ subdomain partitions per dimension ($\alpha^d n_H^d$ subdomains). It is also possible to apply this approach to unstructured meshes and space-time domains. 
 Weak scalability in the sense proposed herein is not only about the capability to solve larger problems but also more complicated ones. E.g., we keep fixed the local Reynolds or P\'eclet number or CFLs, but we increase the global Reynolds or P\'eclet number, facing not only a larger problem but also a harder one, in general. We have used this definition of weak scalability for PDE solvers in the numerical experiments below for time and space-time parallel solvers.  

\subsection{Weak scalability in Time}\label{sec:numexp_TBDDC}

In this case, the spatial domain is not partitioned and only the time integration is distributed through $P_t$ processors. This fact leads to enforced continuity of mean values of the function on the spatial domain $\Omega$ on time interfaces, i.e., constraint \Eq{constraint_time_xt} with $\omega = \Omega$. In order to maintain a constant CFL number, the original time interval $(0,T]$ is scaled with the number of processors, i.e., $T' = P_t T$. As a result, using $P_t$ processors we solve a $P_t$ times larger time domain (and time steps). Note that with this approach neither $| \elt |$ nor $| \sbt |$ are modified through the analysis. 

\subsubsection{Time-parallel Poisson solver} 
Consider the transient Poisson equation (Eq. \Eq{cd} with $\beta = (0,0)$ and $\sigma = 0$) with $\nu=1$ on the unit square spatial domain $\Omega = [0,1]^2$ and $T = 1$. The source term $f$ is chosen such that $u(x, t ) = \sin(\pi x)\sin(\pi y)\sin(\pi t)$ is the solution of the problem. Homogeneous Dirichlet boundary conditions and zero initial condition are imposed. We perform the weak scalability analysis of the TBDDC-GMRES solver with element size $h=\frac{1}{30}$ and several values of $K_n=\frac{| \Delta_n|}{ |\delta_k |} = \{ 10,15,30,60 \}$. 

\begin{figure}[h!!]
\begin{center}
\begin{tabular}{@{}c@{}c@{}}
\includegraphics[width=0.47\textwidth]{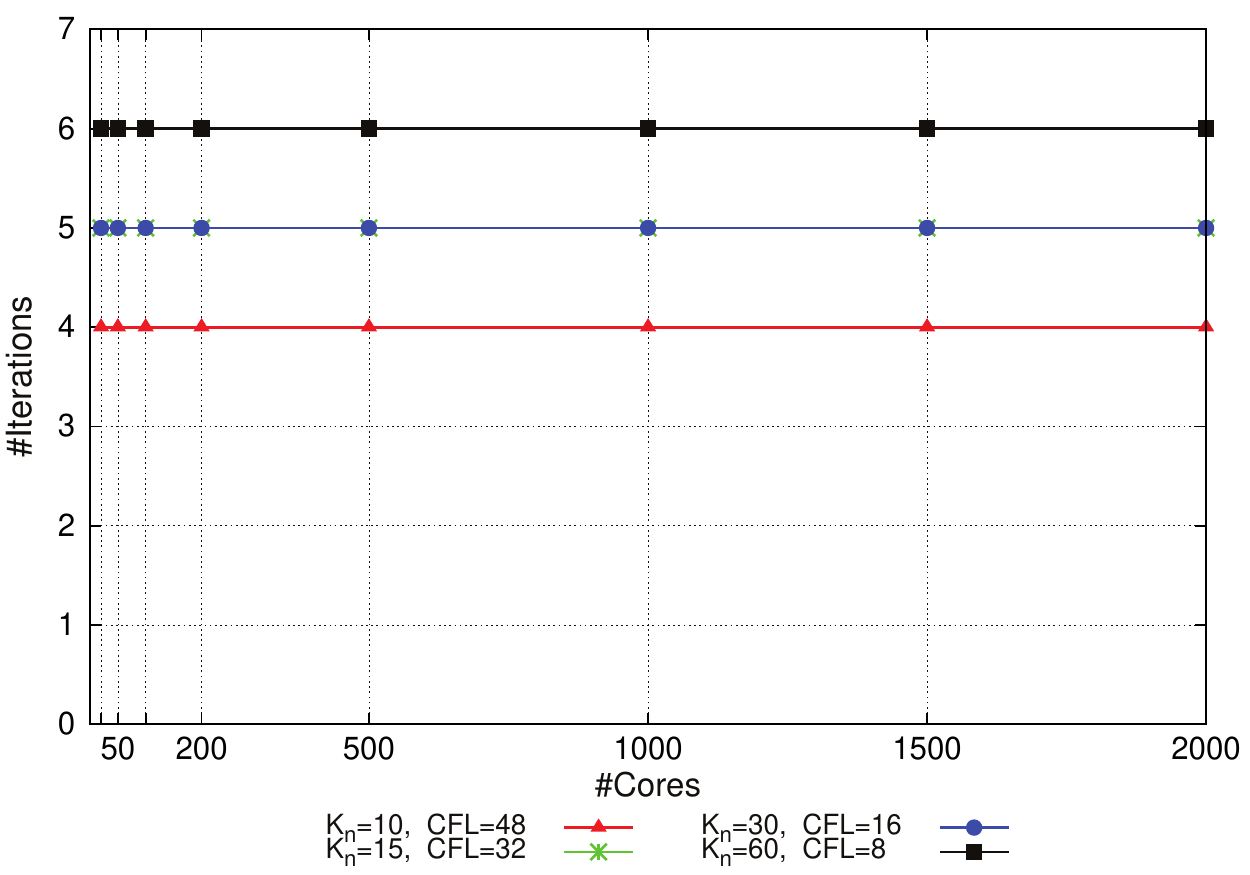} &
\includegraphics[width=0.47\textwidth]{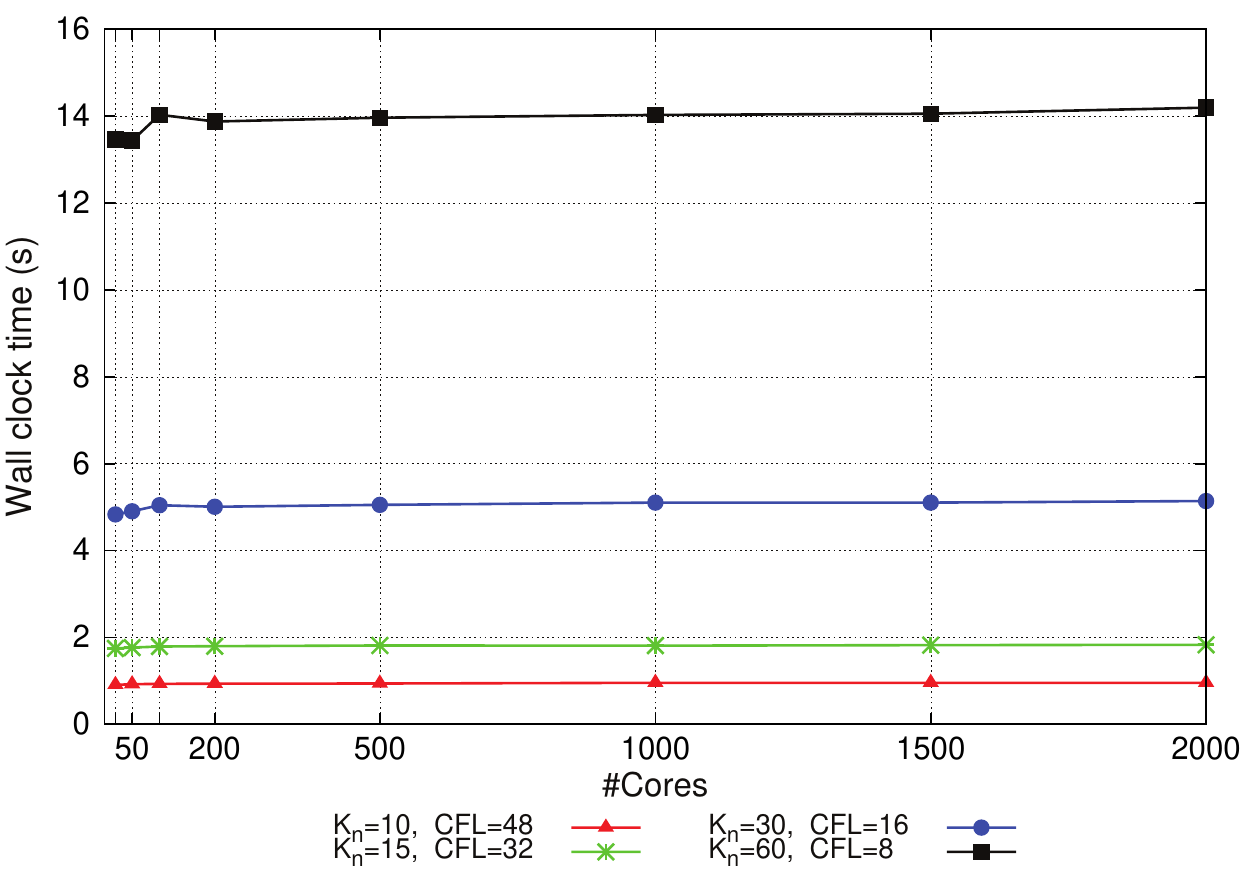} \\
\multicolumn{2}{c}{(a) Iteration counter and computing time for $| \Delta_n | = \frac{8}{15}$. } \\
\includegraphics[width=0.47\textwidth]{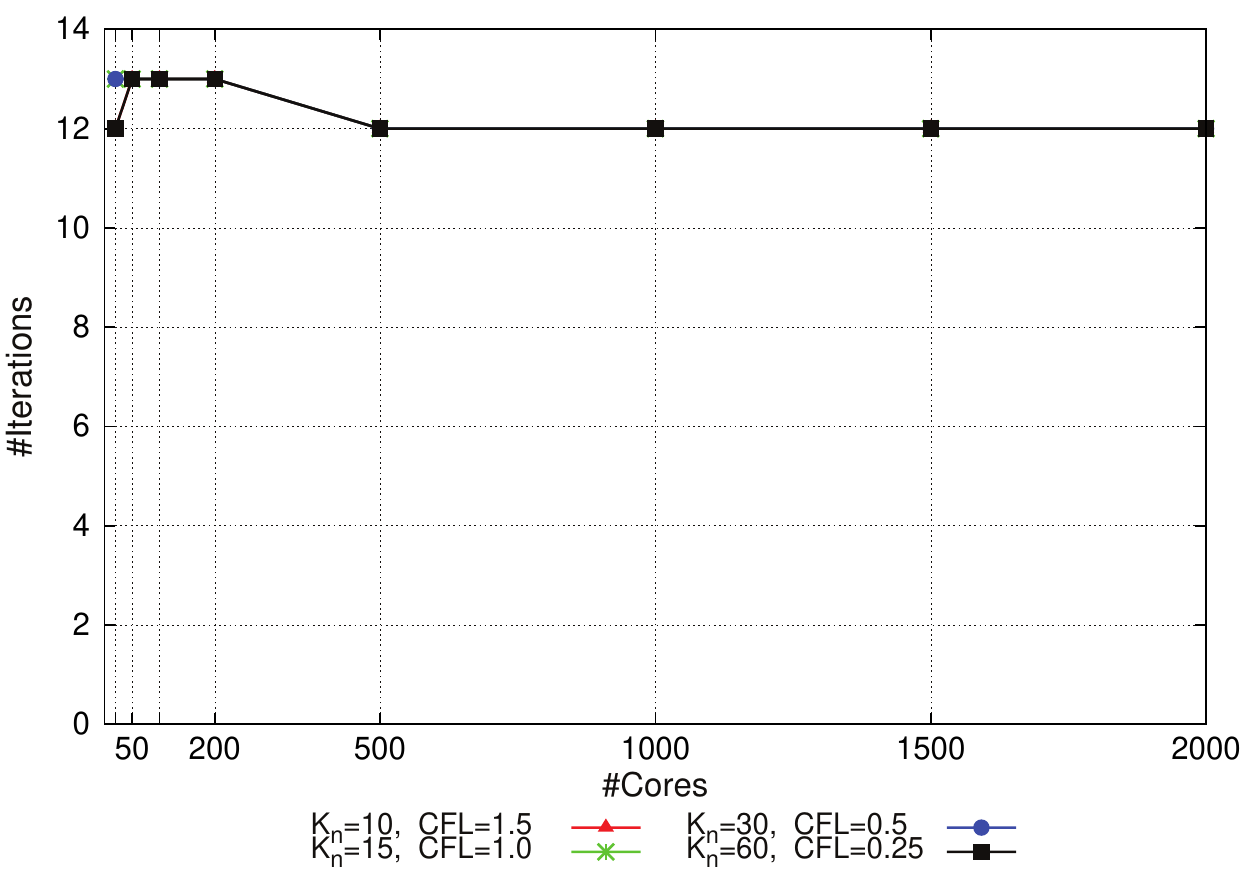} &
\includegraphics[width=0.47\textwidth]{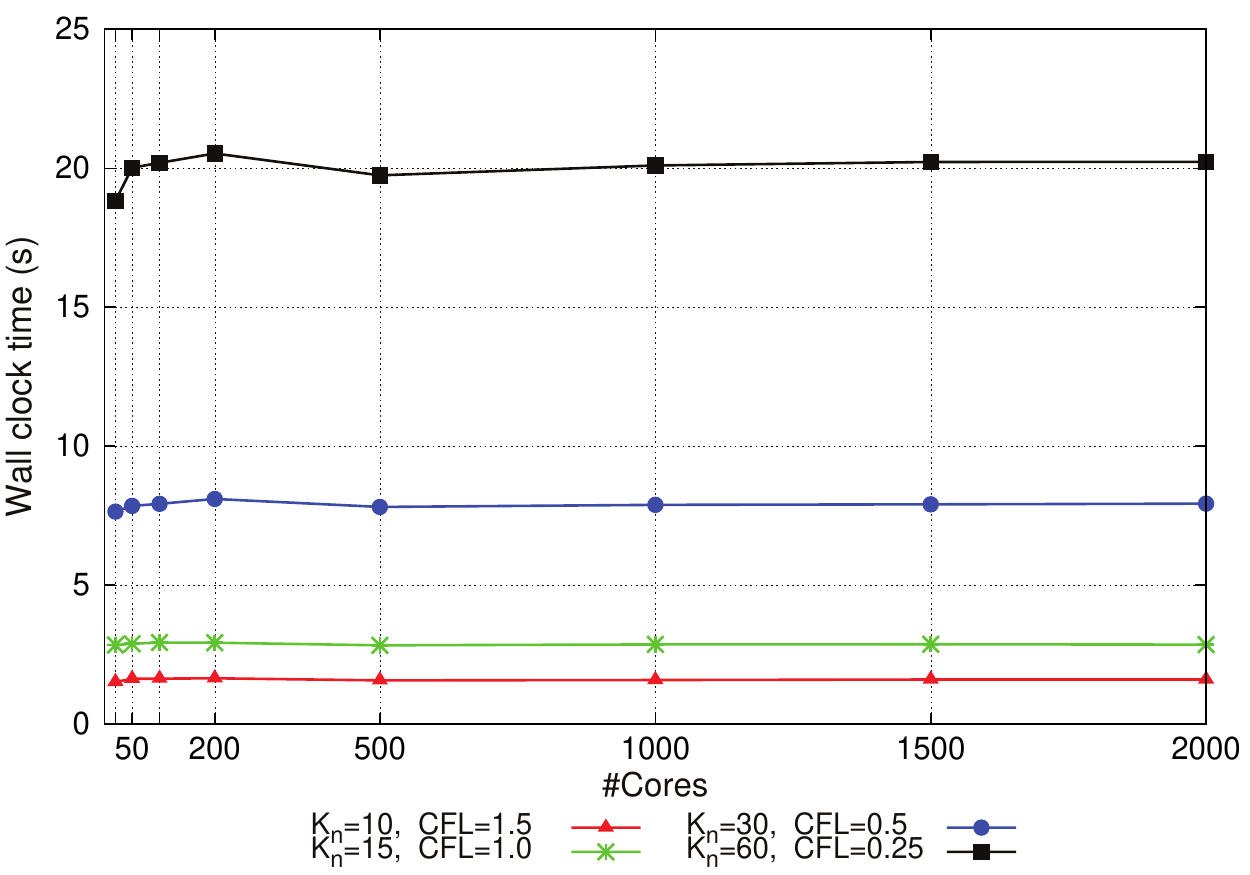} \\
\multicolumn{2}{c}{(b) Iteration counter and computing time for $ | \Delta_n | = \frac{1}{60}$.} \\
\includegraphics[width=0.47\textwidth]{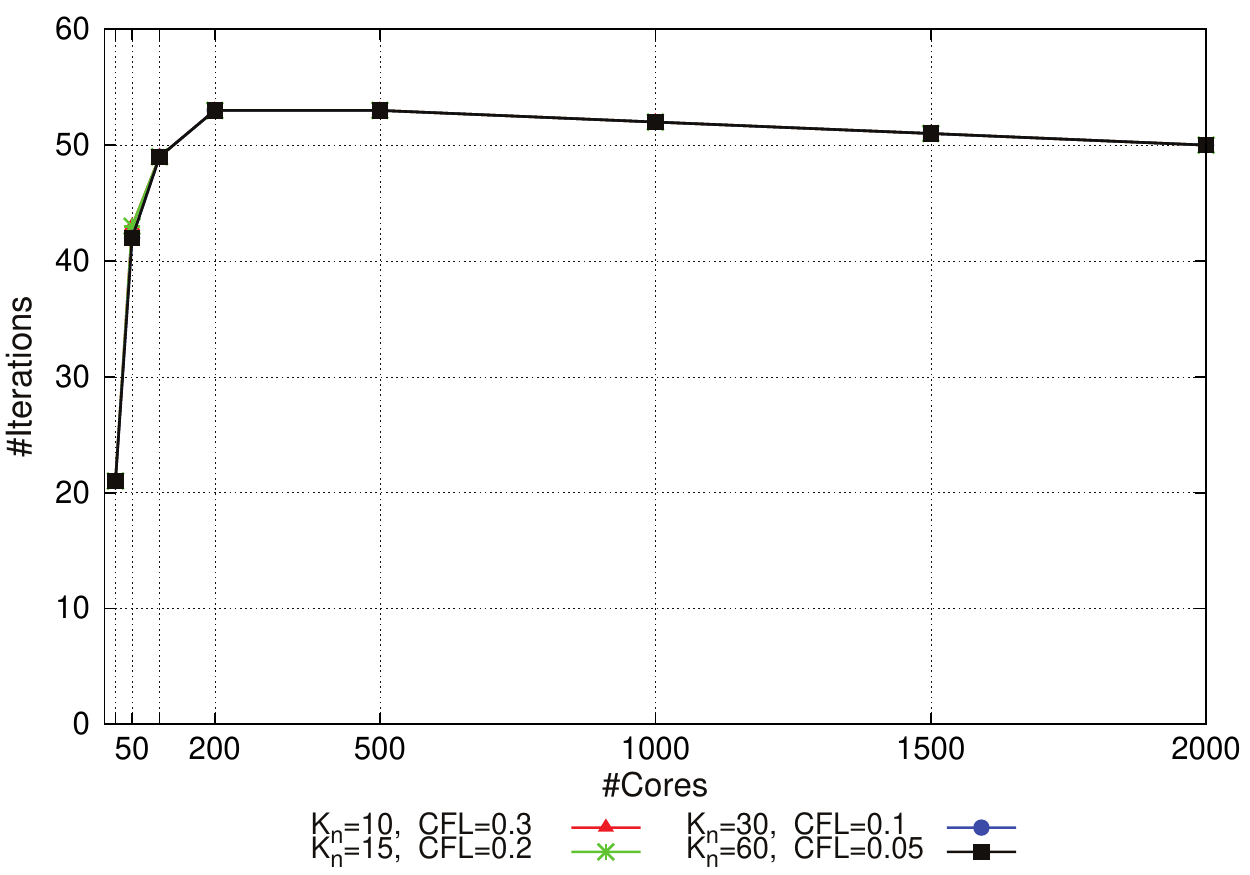} &
\includegraphics[width=0.47\textwidth]{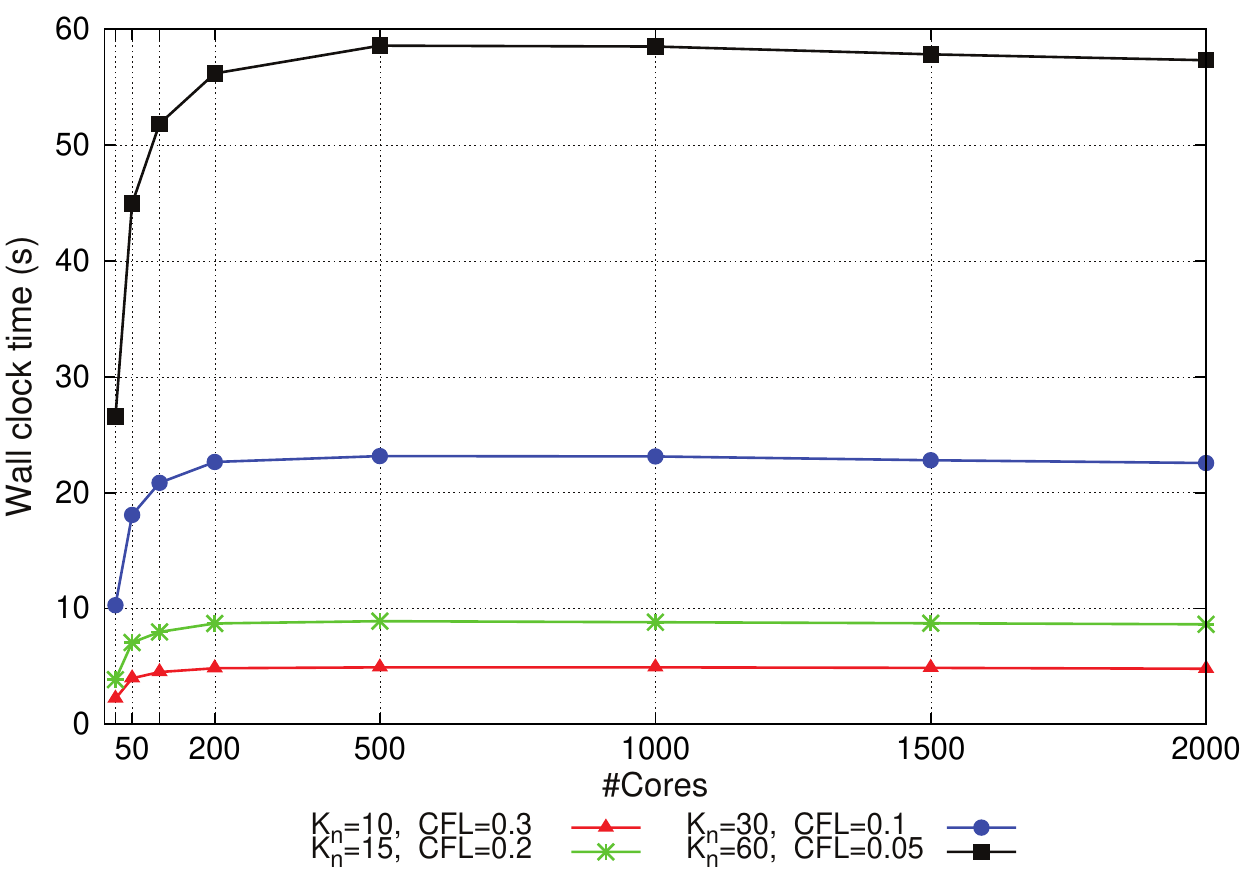} \\
\multicolumn{2}{c}{(c) Iteration counter and computing time for $ | \Delta_n | = \frac{1}{300}$.} \\
\end{tabular}
\end{center}
\caption{Weak scalability for the iterations count (left) and total elapsed time (right) of the TBDDC-GMRES solver in the solution of the unsteady 2D Poisson problem on HLRN. Fixed element size $h=1/30$ while time partition on $P_t$ subdomains. }
\label{fig:heat_tbddc_hlrn}
\end{figure} 

Fig. \ref{fig:heat_tbddc_hlrn} reports the weak scalability analysis of the TBDDC-GMRES solver for this experiment. While $h$ is kept fixed, we evaluate different values of $K_n$  and $| \Delta_n |$, which lead to a wide range of {diffusive} CFLs, shown in \ref{fig:heat_tbddc_hlrn}(a), \ref{fig:heat_tbddc_hlrn}(b), and \ref{fig:heat_tbddc_hlrn}(c) .


Out of these plots, we can draw some conclusions. First, for a fixed local problem size and physical coefficients, reducing the {diffusive} CFL by reducing the time step size results in more iterations. Second, and most salient, the algorithm is in fact weakly scalable. In fact, for a fixed local problem size and {diffusive} CFL, as one increases the number of processors, i.e., computes more time steps, the number of iterations is asymptotically constant. In this range, the overlapped fine/coarse strategy leads to perfect weak scalability for time-parallel solvers too. As a result, this analysis shows the capability of the method to compute X times more time steps with X times more cores for the same {total elapsed} time, which is the main motivation of time-parallelism. 


\subsubsection{Time-parallel CDR solver} Consider now the transient CDR equation \Eq{cd} with $\nu=10^{-3}$ and $\sigma=10^{-4}$ on $\Omega=[0,1]^2$ with homogeneous Dirichlet boundary conditions and null initial solution. We take $T= \frac{1}{10}$. In order to show results for different {convective} CFL ranges, two convection velocity fields are analysed, namely $\beta = (1,0)$ and $\beta= (10,0)$. The CFL values shown are those corresponding to the convective term since it is more restrictive than the diffusive CFL in all the cases being considered. The SUPG stabilization technique is employed (see \cite{brooks_streamline_1982}). 

We perform the study with several values of $K_n=\frac{| \Delta_n |}{| \delta_k |} = \{ 10, 30 ,60 \} $, which lead to different {convective} CFLs. For the first case (Fig. \ref{fig:CD_tbddc_analytical_hlrn}), the source term is chosen such that the function $u(x, t ) = \sin(\pi x)\sin(\pi y)\sin(\pi t)$ is the solution of the problem. For the second test (Fig. \ref{fig:boundary_layer_hlrn}), we take  ${f}={1}$, with a boundary layer formation. \\
Out of these plots we can extract the same conclusions from Figs. \ref{fig:CD_tbddc_analytical_hlrn} and \ref{fig:boundary_layer_hlrn} as for the Poisson problem. The method is weakly scalable in time for transient CDR problems.

\begin{figure}[h!]
\begin{center}
\begin{tabular}{@{}c@{}c@{}}
\includegraphics[width=0.47\textwidth]{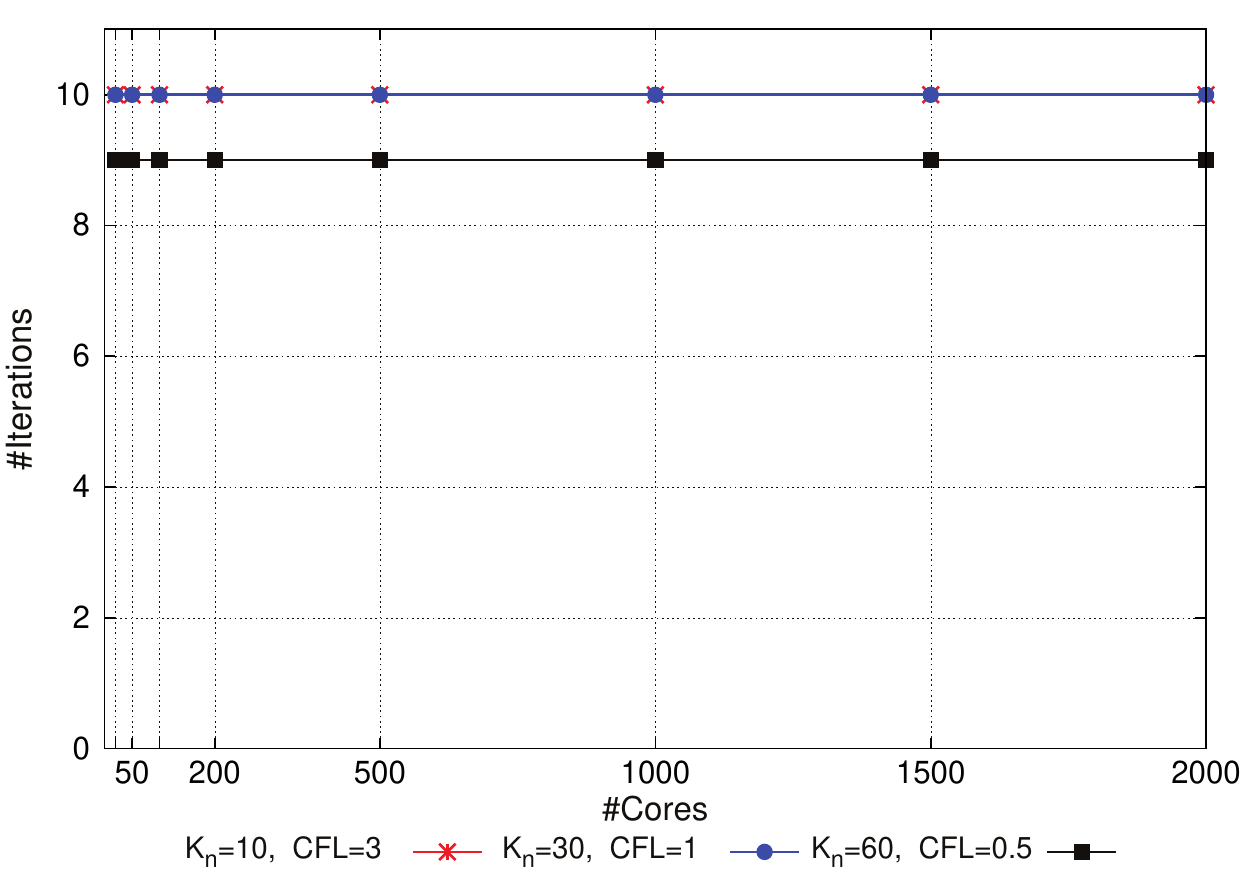} &
\includegraphics[width=0.47\textwidth]{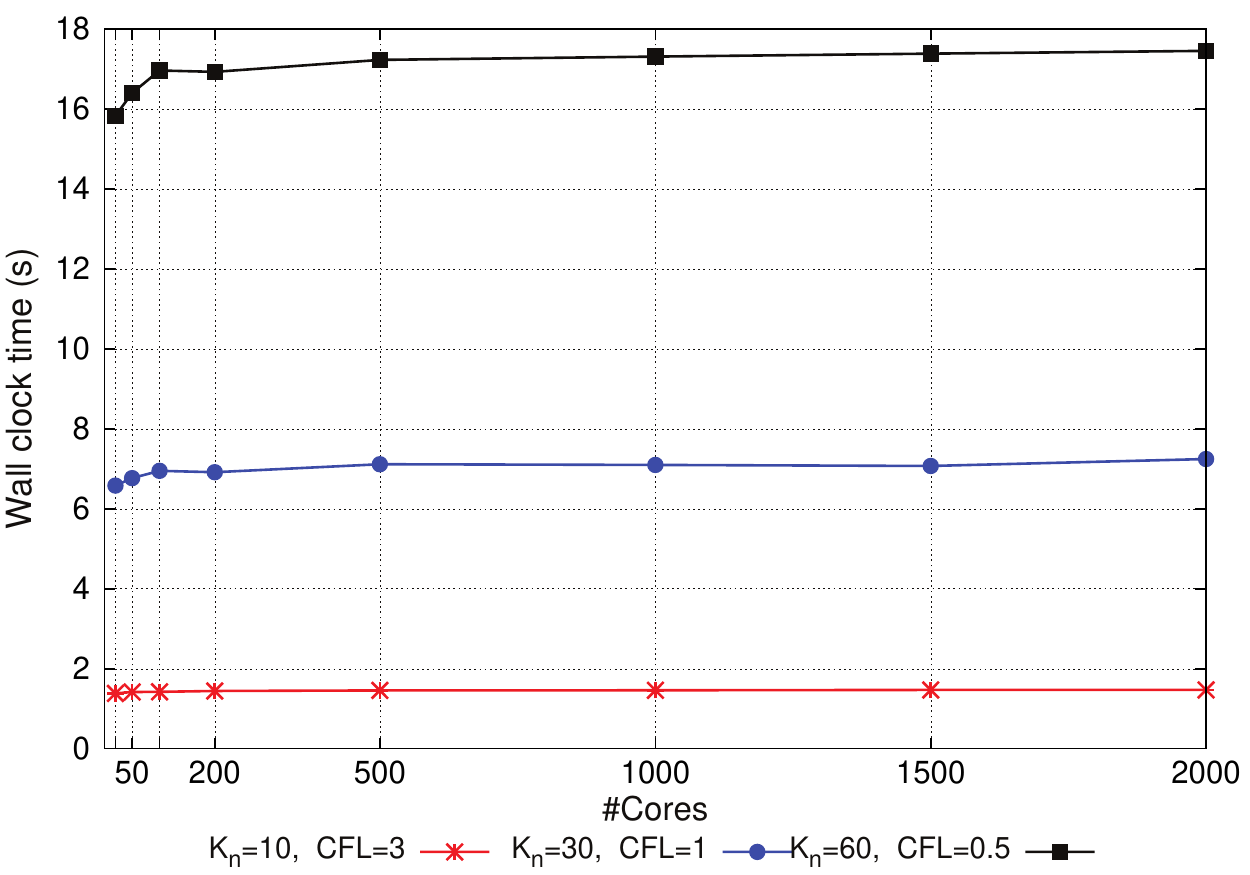} \\
\multicolumn{2}{c}{(a) Iteration counter and computing time with $\beta= (10,0)$. } \\
\includegraphics[width=0.47\textwidth]{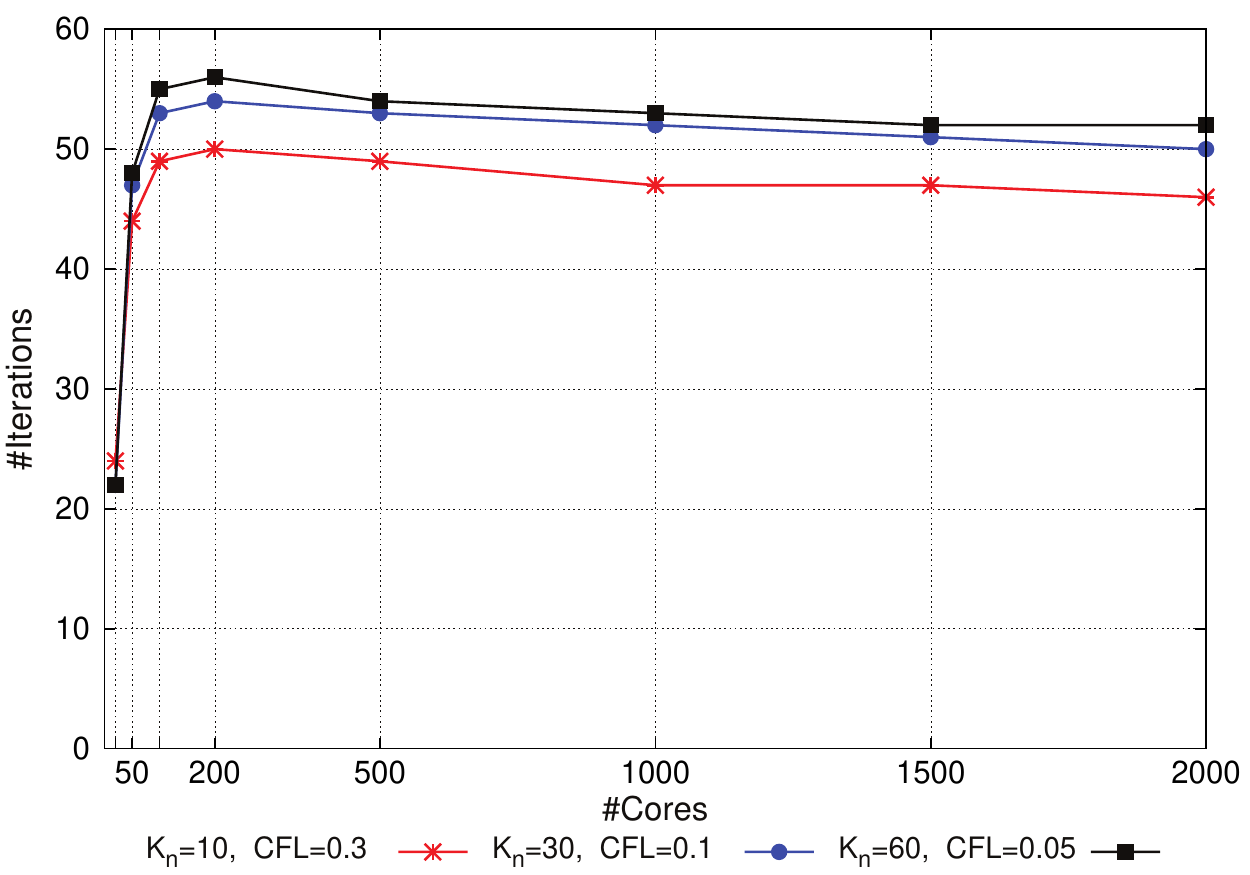} &
\includegraphics[width=0.47\textwidth]{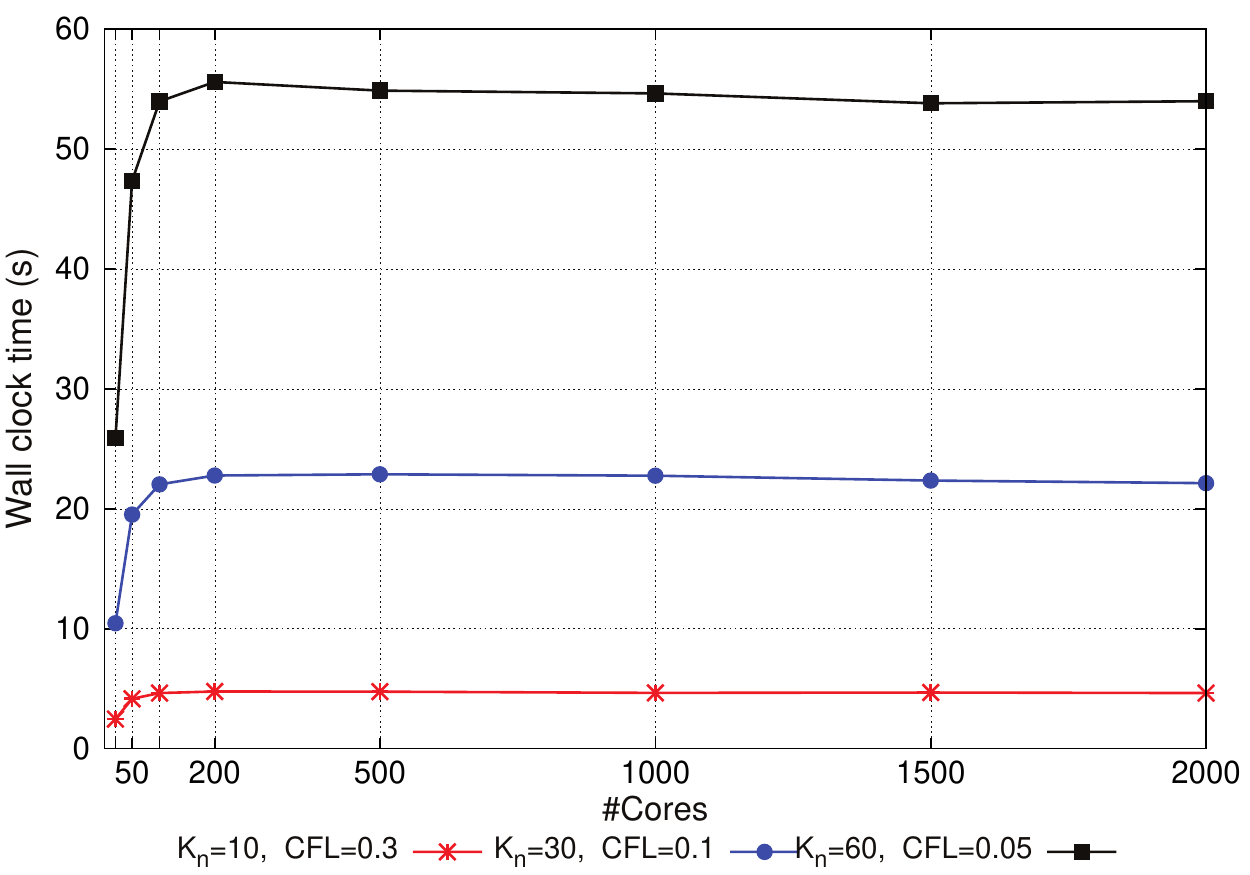} \\
\multicolumn{2}{c}{(b) Iteration counter and computing time with $\beta=(1,0)$. } \\
\end{tabular}
\end{center}
\caption{Weak scalability for the total elapsed time (right) and number of GMRES iterations (left) of the TBDDC solver in the solution of the 2D CDR equation on HLRN (sinusoidal solution). Fixed element size to $h=1/30$. }
\label{fig:CD_tbddc_analytical_hlrn}
\end{figure}


\begin{figure}[h!]
\begin{center}
\begin{tabular}{@{}c@{}c@{}}
\includegraphics[width=0.47\textwidth]{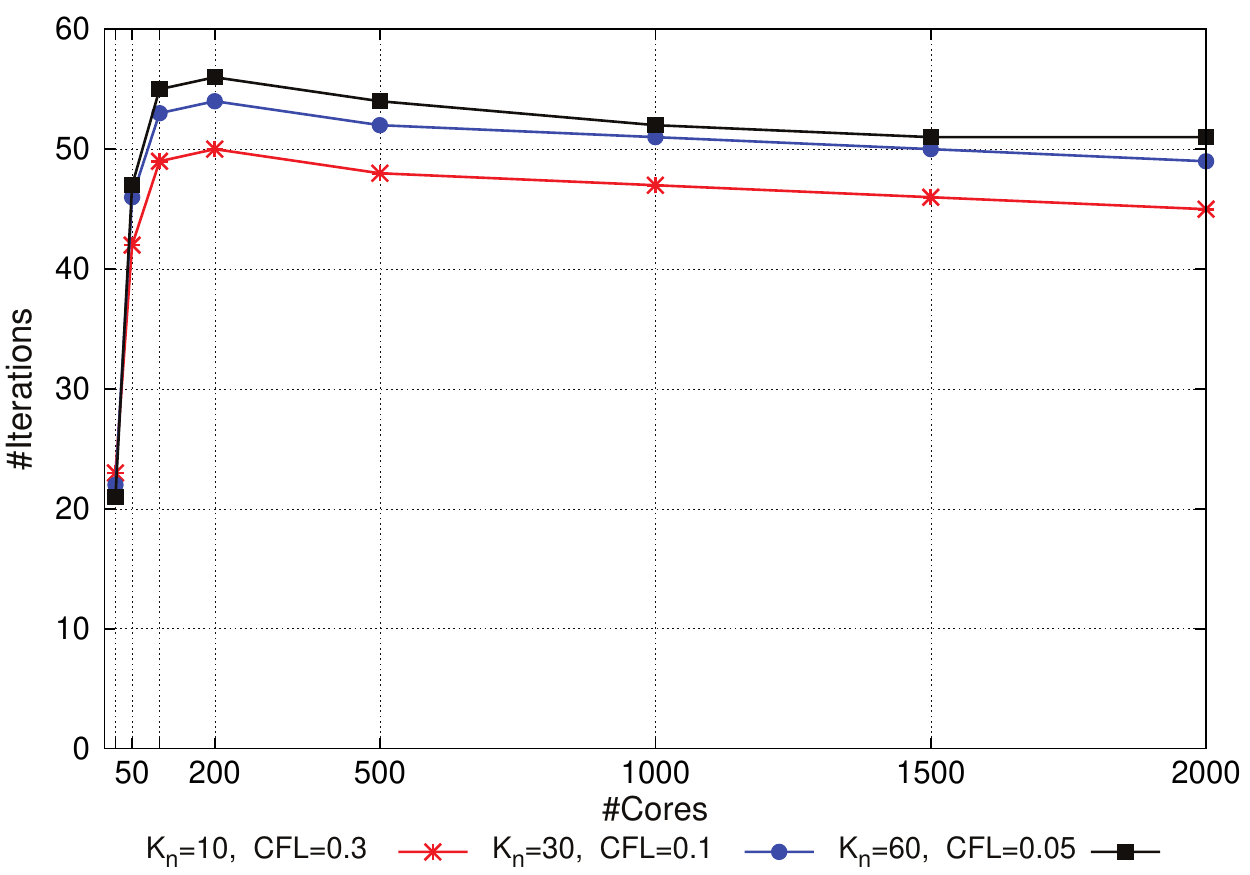} &
\includegraphics[width=0.47\textwidth]{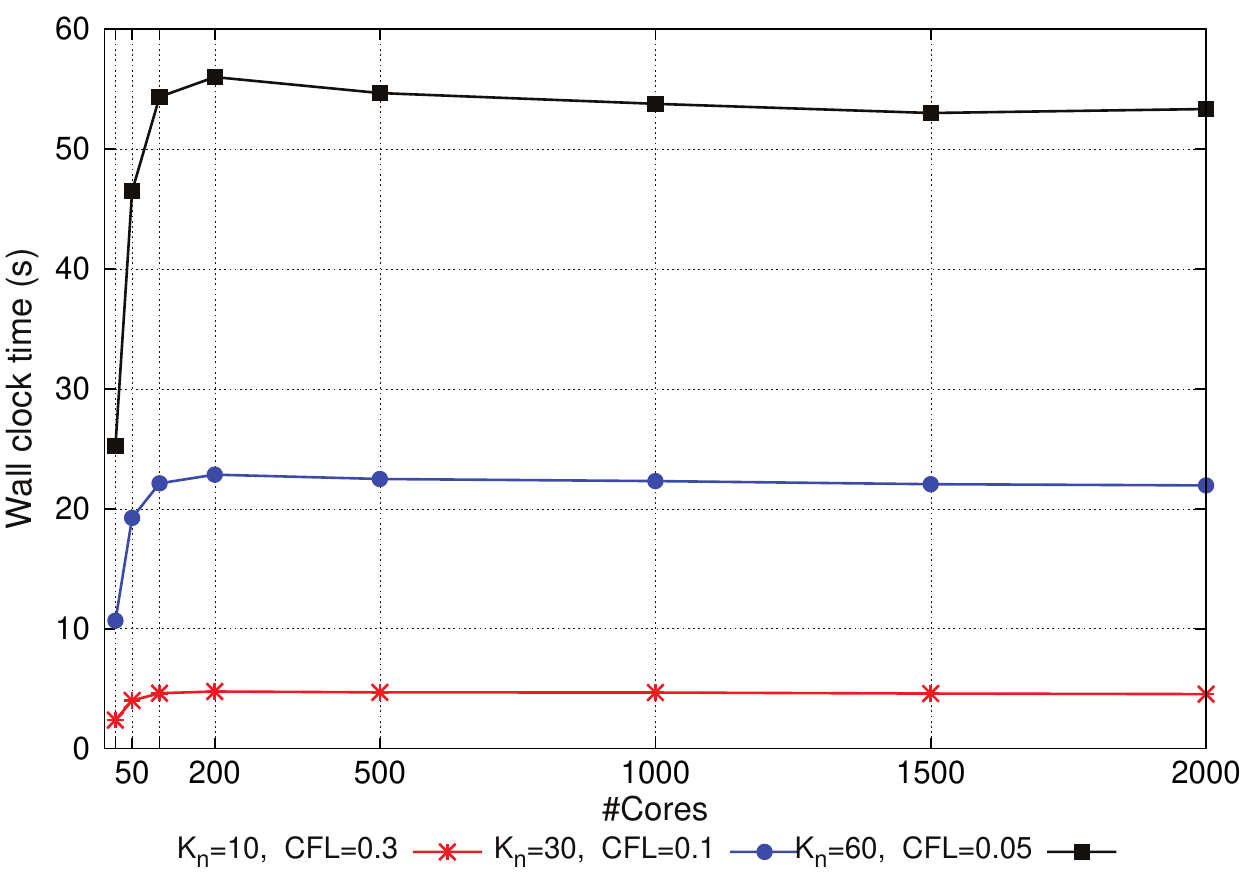} \\
\end{tabular}
\end{center}
\caption{Weak scalability for the total elapsed time (right) and number of GMRES iterations (left) of the TBDDC solver in the solution of the 2D CDR equation (boundary layer) on HLRN. Fixed element size to $h=1/30$.}
\label{fig:boundary_layer_hlrn}
\end{figure}

\subsection{ Weak scalability in space-time}\label{sec:numexp-STBDDC}
The STBDDC preconditioner is considered here, for the set of constraints \Eq{constraint_space_xt}-\Eq{coarse-dofs2}. In this case, the spatial domain $\Omega$ is scaled by $P_x$ and $P_y$ in the corresponding directions, where $P_x = P_y$ in all  cases. On the other hand, the time interval $(0,T]$ is scaled by $P_t$, leading to a $P = P_x \times P_y \times P_t$ partition of the scaled space-time domain $P_x \Omega \times P_t (0,T]$. Therefore, the relative weight of the operators is kept constant through a weak scalability analysis. Local problem loads will be given by $\frac{H}{h}$ in space and $K_n = \frac{| \Delta_n |}{| \delta_k |}$ in time, i.e., $(\frac{H}{h})^d \times \frac{| \Delta_n |}{| \delta_k |}$. 


\subsubsection{Space-time Poisson solver}\label{subsec-st_heat_results} Consider the Poisson problem  (Eq. \Eq{cd} with $\beta = (0,0)$ and $\sigma = 0$) with $\nu=1$, $f=1$, and homogeneous Dirichlet boundary conditions and zero initial condition. The original spatial domain is $\Omega =[0,1]^2$ while the original time domain is $(0,0.1]$.  

\begin{figure}[h!]
\begin{center}
\begin{tabular}{@{}c@{}c@{}}
\includegraphics[width=0.47\textwidth]{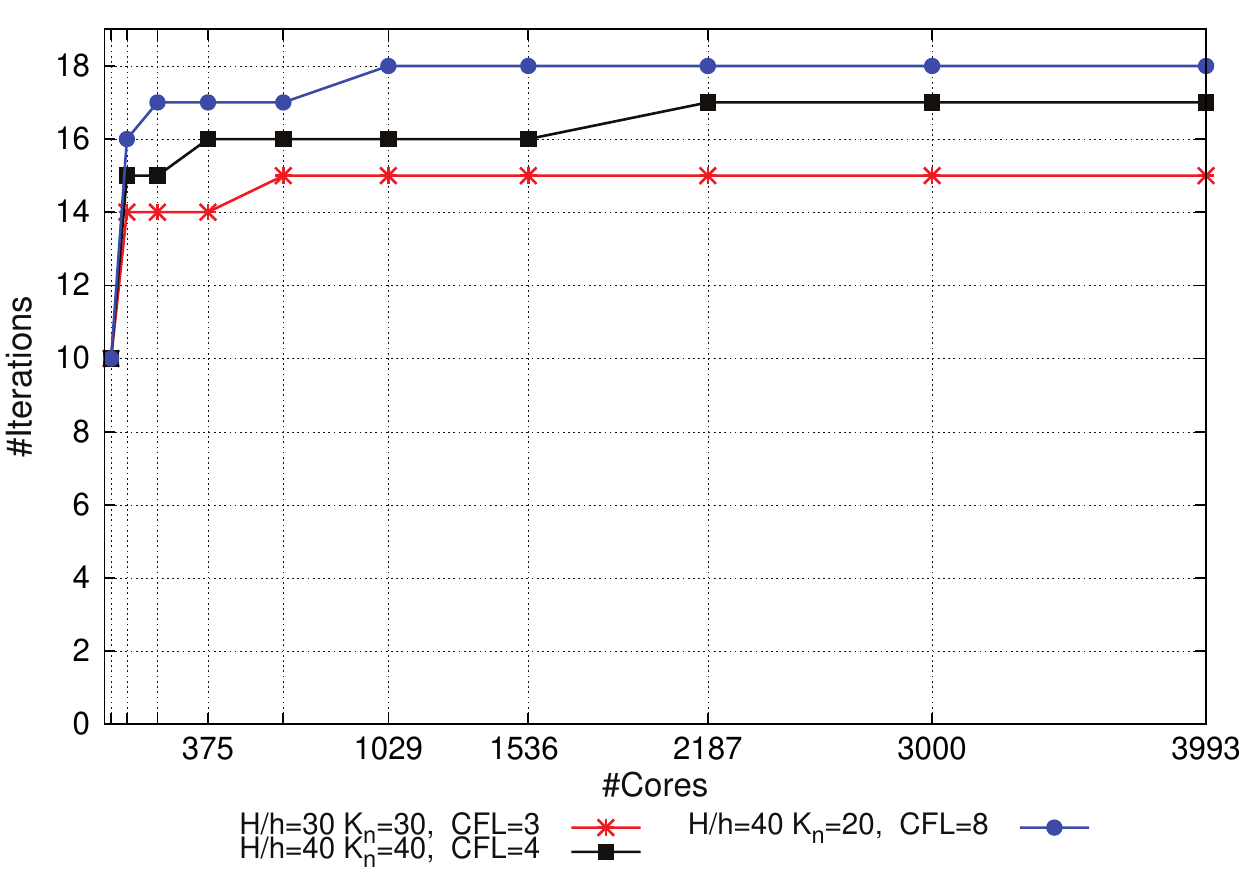} &
\includegraphics[width=0.47\textwidth]{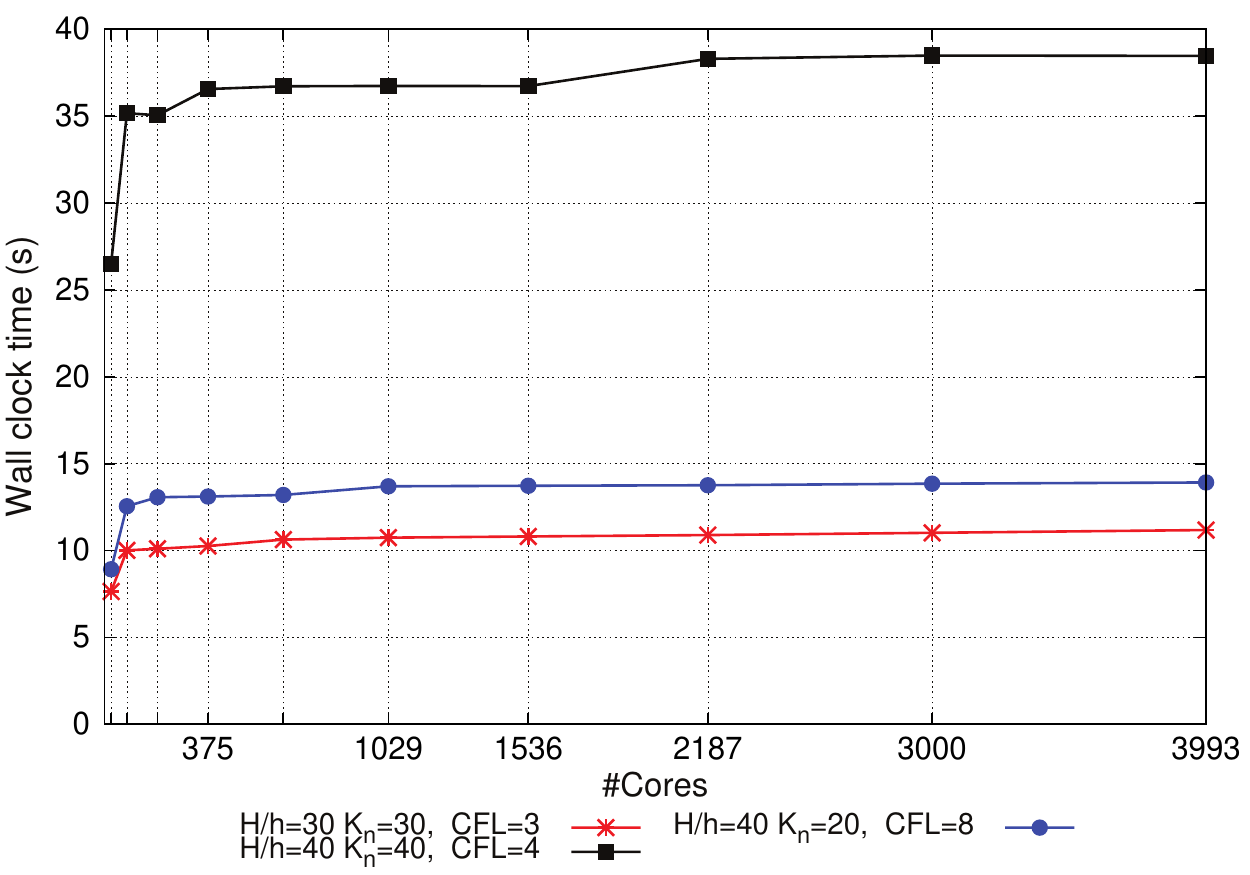} \\
\multicolumn{2}{c}{(a) Iteration counter and wall clock times for $P_t=3P_x$ } \\
\includegraphics[width=0.47\textwidth]{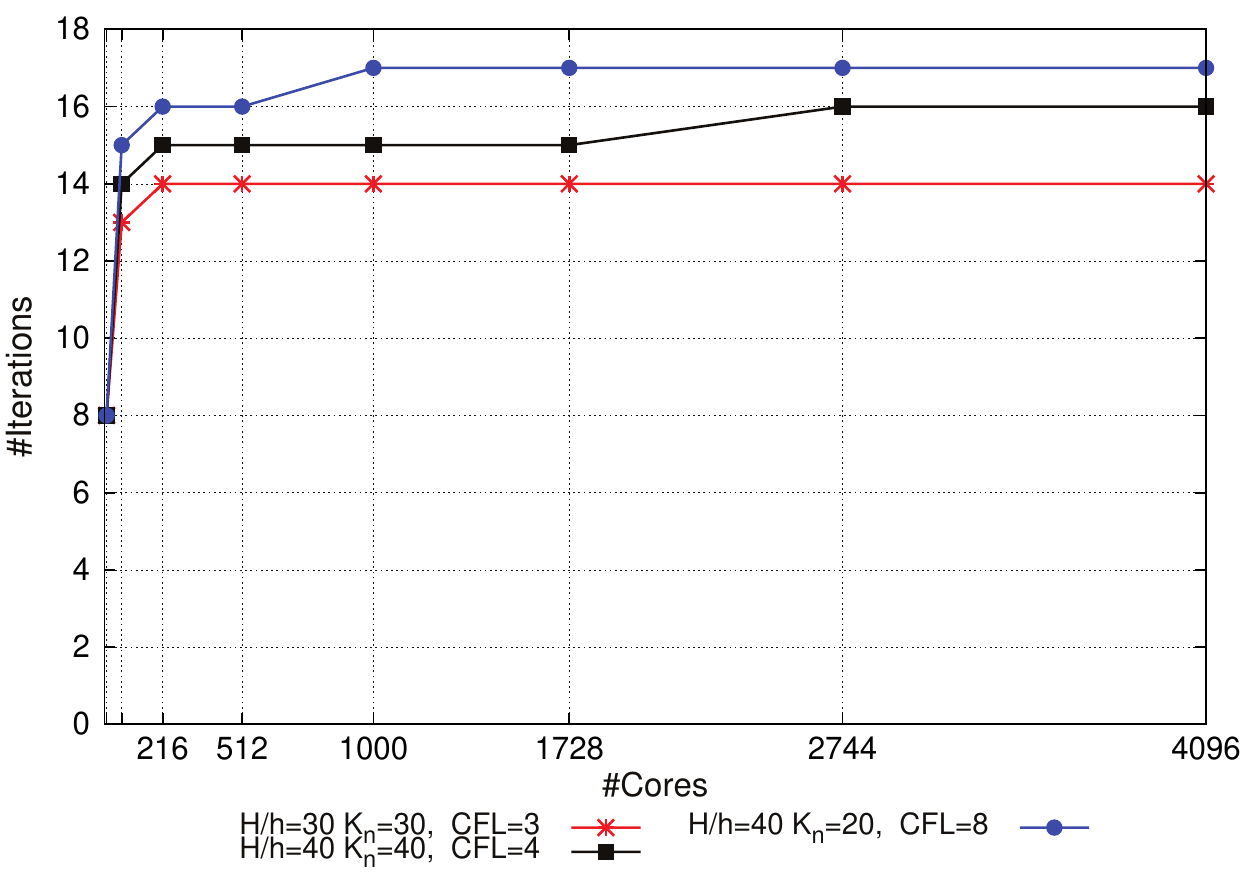} &
\includegraphics[width=0.47\textwidth]{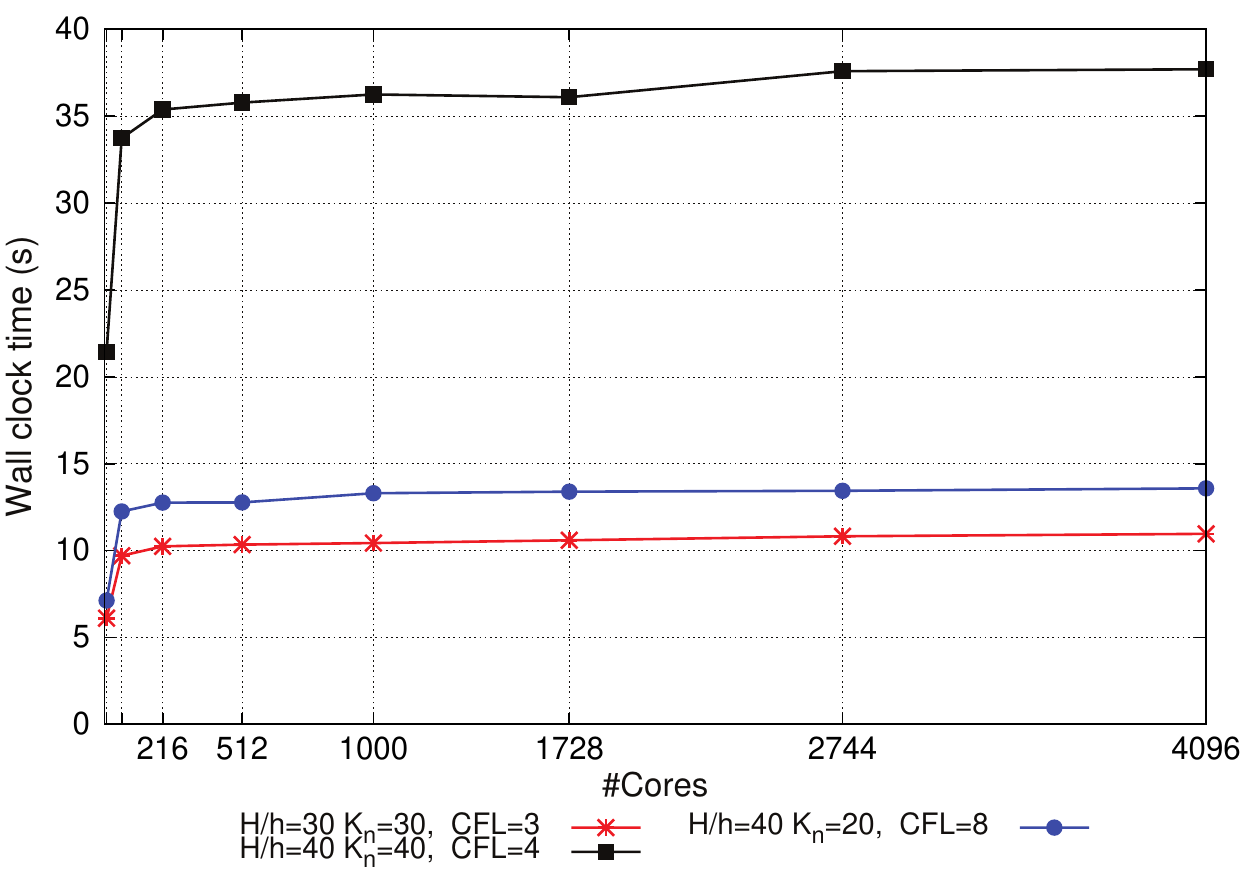} \\
\multicolumn{2}{c}{(b) Iteration counter and wall clock times for $P_t=P_x$ } \\
\includegraphics[width=0.47\textwidth]{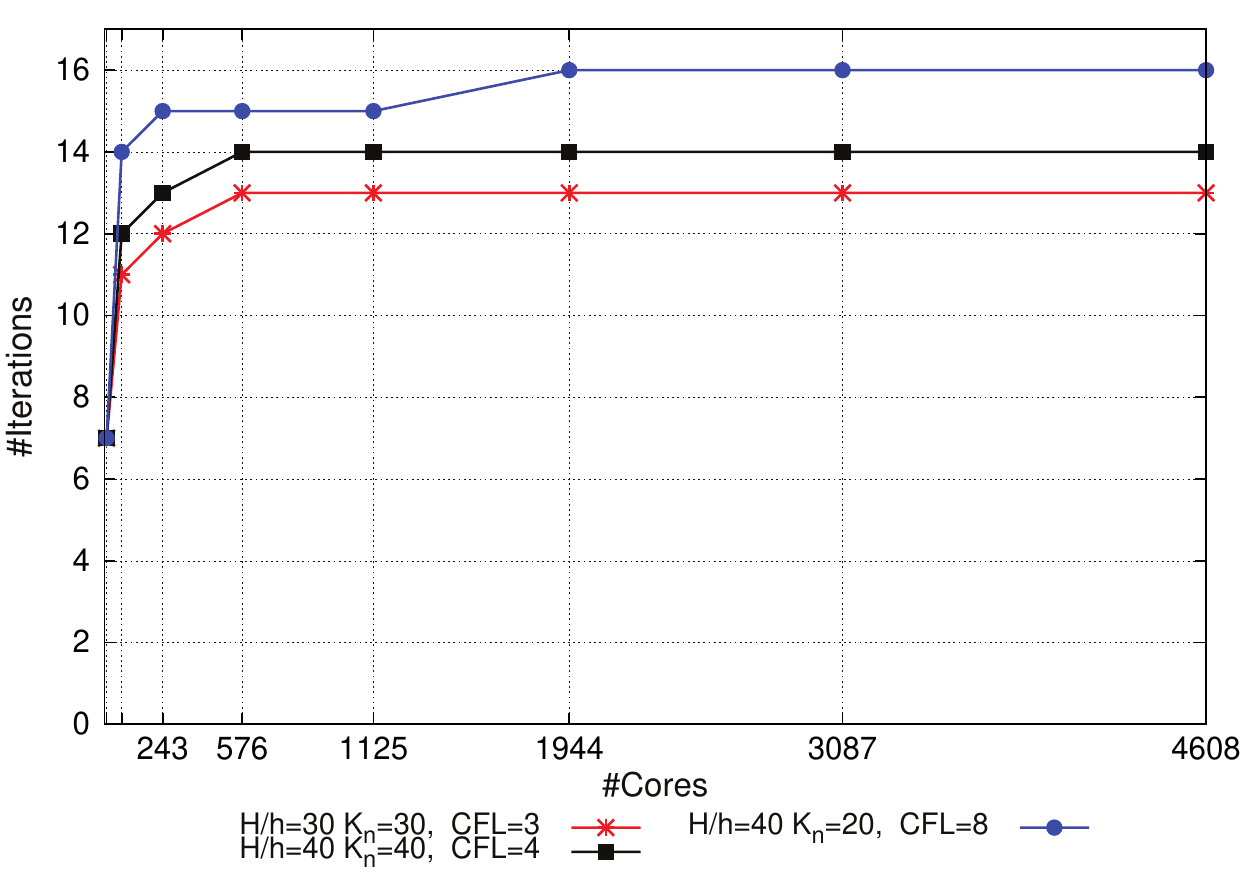} &
\includegraphics[width=0.47\textwidth]{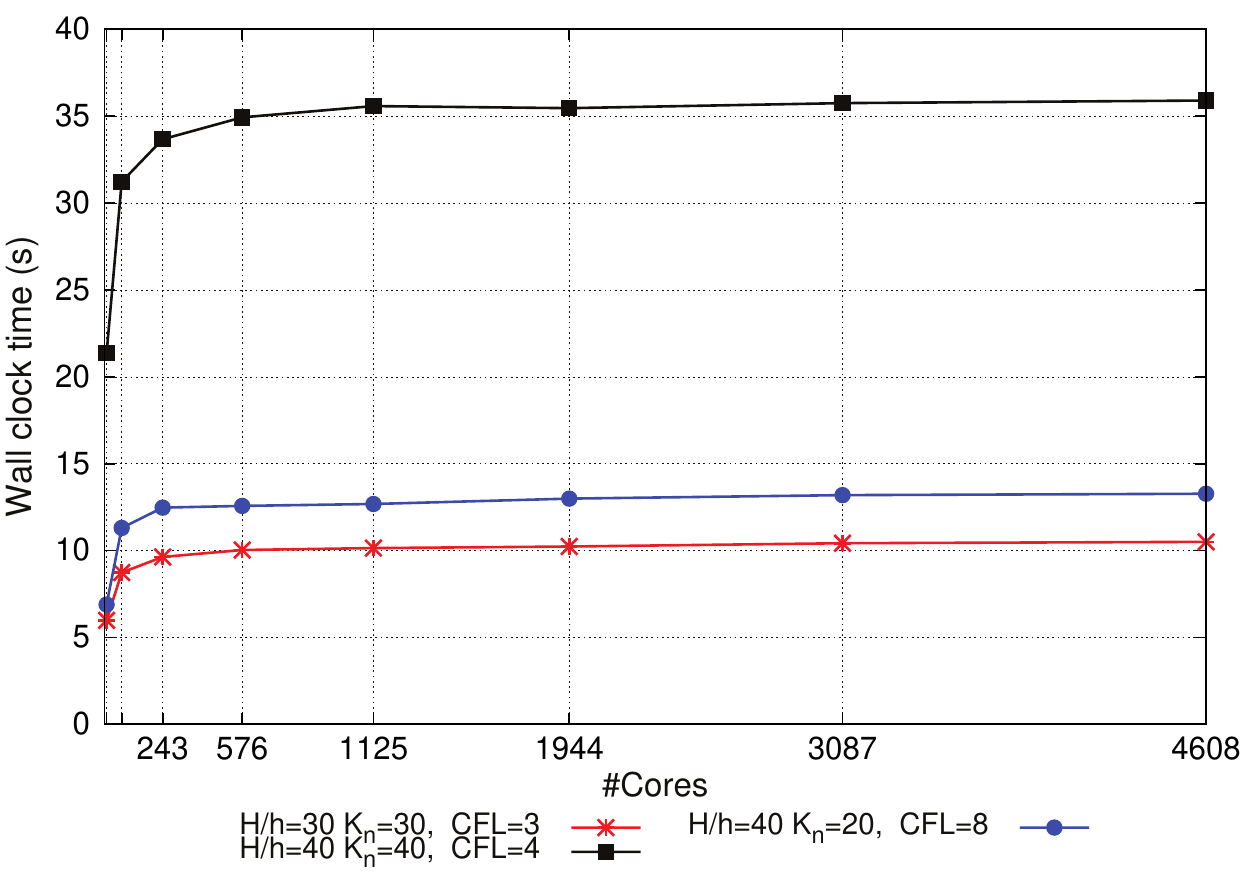} \\
\multicolumn{2}{c}{(c) Iteration counter and wall clock times for $P_x=3P_t$ } \\
\end{tabular}
\end{center}
\caption{Weak scalability for the total elapsed time (right) and number of GMRES iterations (left) of the STBDDC solver in the solution of the 2D Poisson problem on HLRN. }
\label{fig:heat_stbddc_hlrn}
\end{figure}

\begin{figure}[h!]
\begin{center}
\begin{tabular}{@{}c@{}c@{}}
\includegraphics[width=0.47\textwidth]{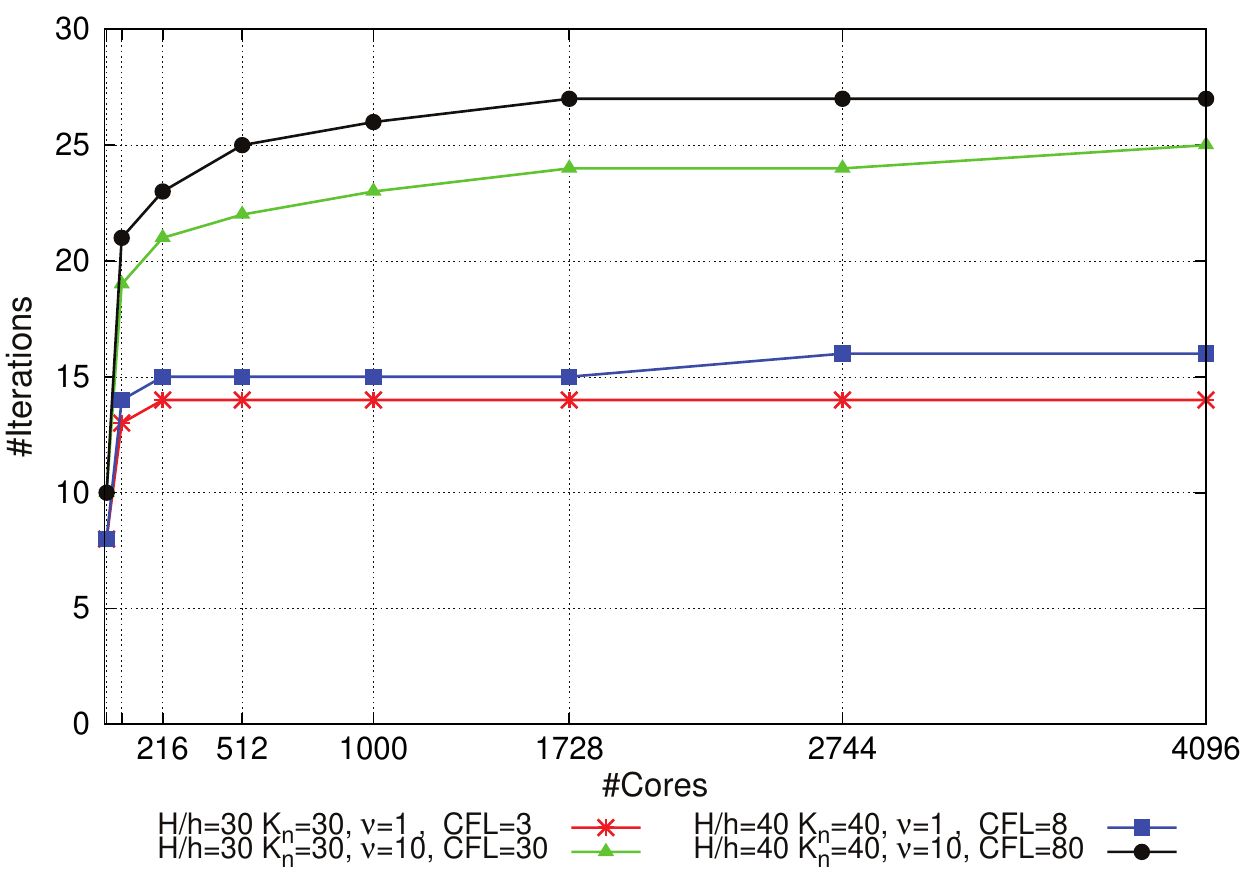} &
\includegraphics[width=0.47\textwidth]{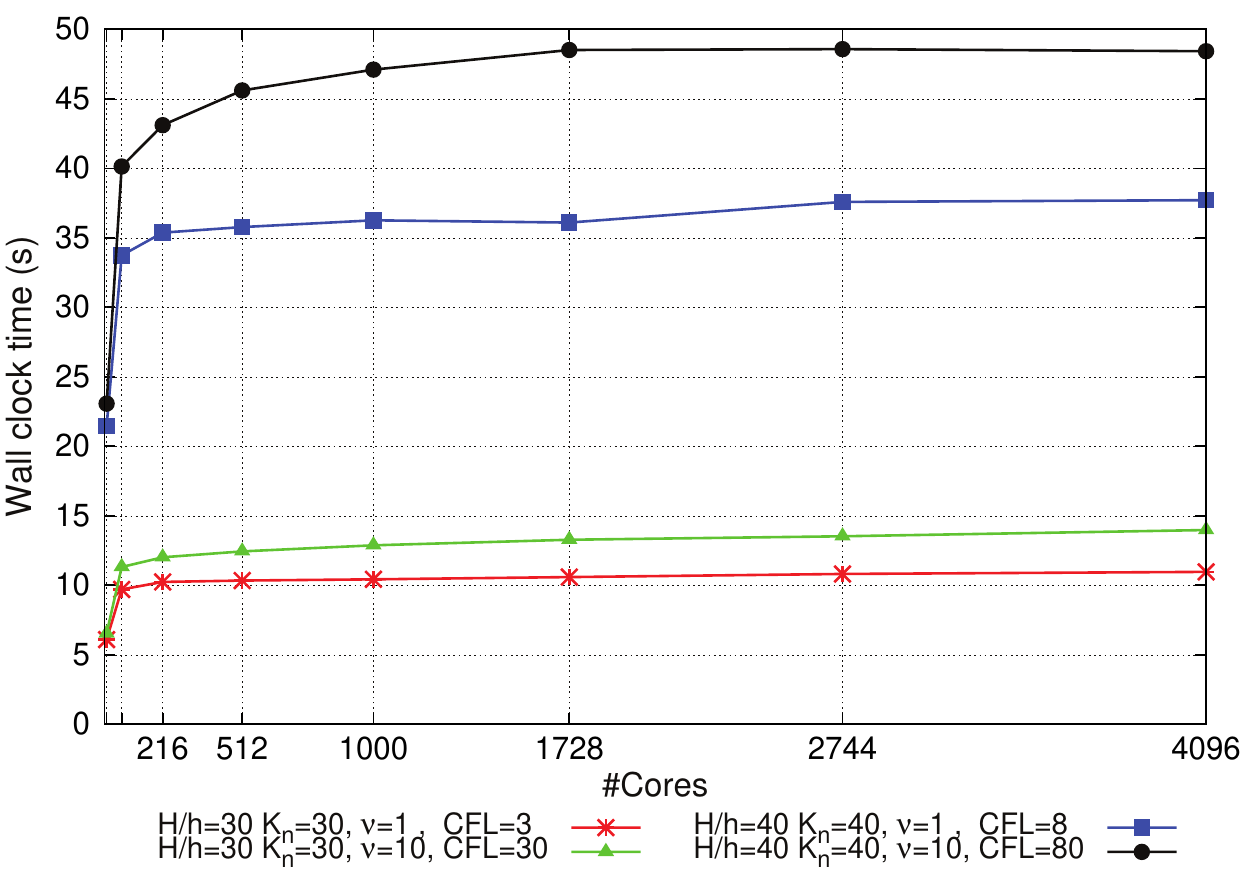} \\
\end{tabular}
\end{center}
\caption{Weak scalability for the total elapsed time (right) and number of GMRES iterations (left) of the STBDDC solver in the solution of the 2D Poisson problem on HLRN. Partition is done equally in time and space, i.e., $P_x=P_t$. }
\label{fig:heat_stbddc_nu}
\end{figure}

Fig. \ref{fig:heat_stbddc_hlrn} shows weak scalability results for the STBDDC-GMRES solver. Number of iterations (left) and {total elapsed times} (right) have been reported for three different ratios between spatial and time partitions $\frac{P_x}{P_t}$ in \ref{fig:heat_stbddc_hlrn}(a), \ref{fig:heat_stbddc_hlrn}(b), and  \ref{fig:heat_stbddc_hlrn}(c). Also, at every figure, three different local problem sizes, and thus {diffusive} CFLs, have been considered.

The most salient information from these figures is the fact that the scheme is also scalable in space-time simulations. Here, one is not only solving a larger problem in time (as above) but also in space. This result is not surprising since the spatial BDDC preconditioner is known to be weakly scalable and the time-parallel version has also been proved to be weakly scalable in Sect. \ref{sec:numexp_TBDDC}. The influence of $\frac{P_x}{P_t}$ on the number of iterations is very mild; also the effect of the {diffusive} CFL is mild in this case. The overlapping strategy is fully effective in the range under consideration, because perfect weak scalability can be observed. The effect of the {diffusive} CFL for a fixed local problem size, obtained by multiplying by 10 the viscosity, is reported in Fig. \ref{fig:heat_stbddc_nu}. In this case, a larger {diffusive} CFL leads to more iterations but weak scalability is also achieved.

Next, we want to compare the space-time solver against a sequential-in-time approach. We fix the time step size to $ | \elt | =10^{-3} $ and $\frac{H}{h}=30$. Thus, the time interval is $T = K | \elt |$, when considering $K$ time steps. The sequential approach makes use of $P_{s} = P_x \times P_y = 4^2$ processors (space-parallelism only) to solve recursively the spatial problem for increasing values of $K$. 
The space-time approach is using $P_{s} \times P_t = 4^2 P_t$ processors to solve the same problem, with a local number of time steps $K_n=10$. The motivation for such analysis is to assess the benefit of time-parallelism in linear problems when spatial parallelism cannot be further exploited efficiently due to very low load per processor. 
\begin{figure}[h!]
\begin{center}
\begin{tabular}{@{}c@{}c@{}}
\includegraphics[width=0.47\textwidth]{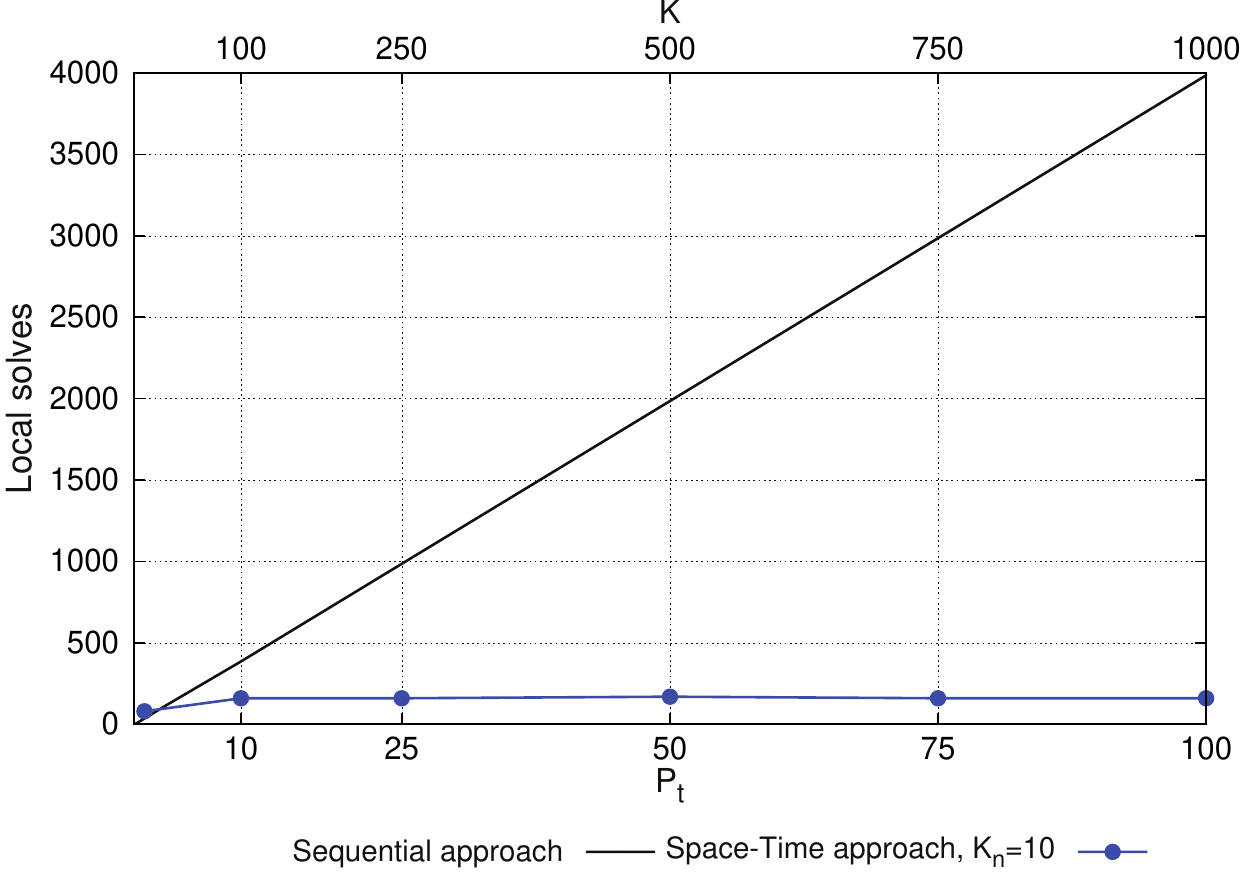} &
\includegraphics[width=0.47\textwidth]{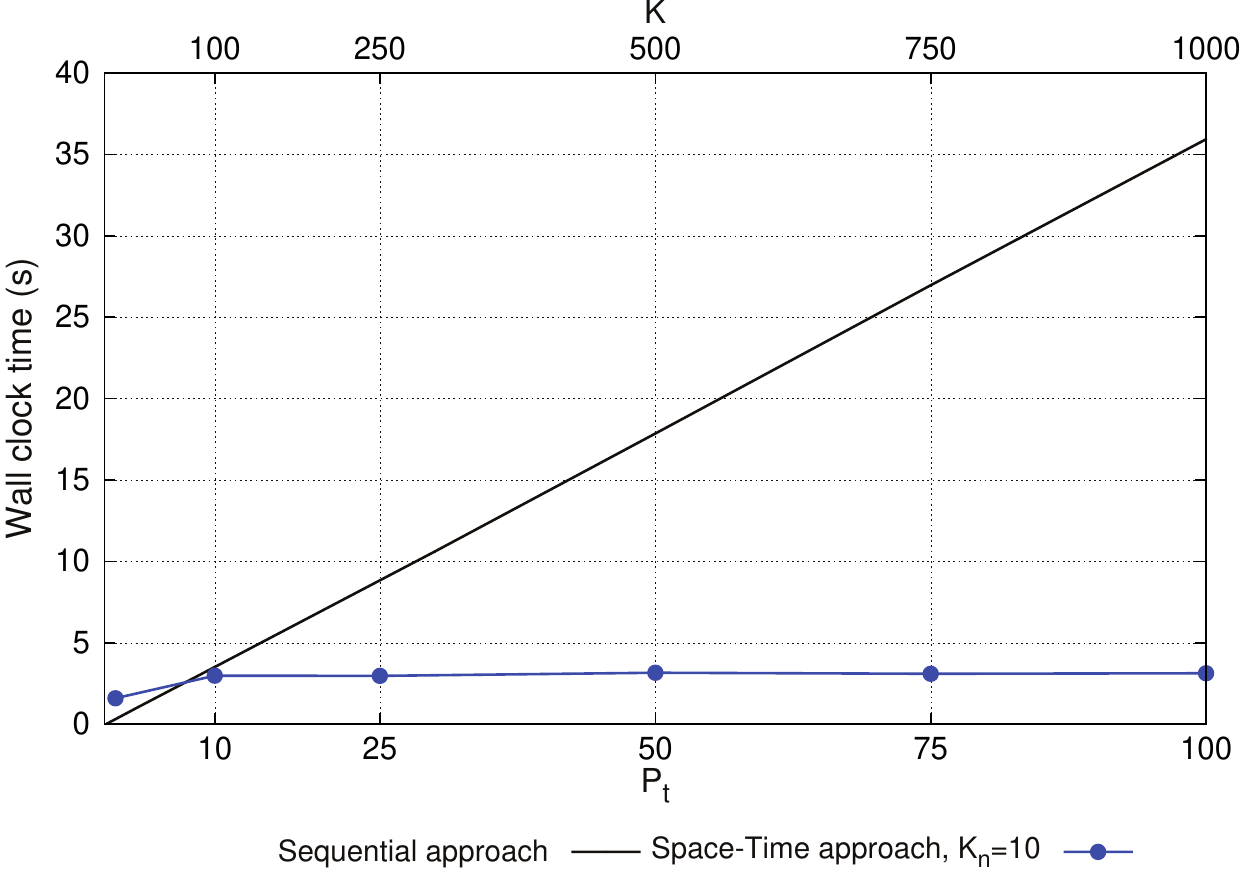} \\
\end{tabular}
\end{center}
\caption{ Comparison between the sequential and space-time solvers for the transient Poisson problem on MareNostrum supercomputer.  Spatial partition is fixed to $P_x=P_y=4$. The space-time approach is using $P=4^2 P_t$ MPI fine-level tasks. }
\label{fig:laplacian_st_vs_eq}
\end{figure}

Figure \ref{fig:laplacian_st_vs_eq} shows the comparison between the sequential approach  and the STBDDC preconditioned space-time solver up to $K=1000$ time steps. The theoretical cost of the sequential approach is proportional to $K$ steps times the {elapsed} time of the $A^{-1}_\omega$ local solve (a preconditioned GMRES iteration). On the other hand, the theoretical cost of the space-time solver is proportional to the cost of the $\boldsymbol{A}^{-1}_\sbxt$ local space-time solves, and in turn, $K_n$ times spatial $A^{-1}_\omega$ local solves (exploiting locally sequentiality in time). The number of local solves is plotted for both sequential and space-time approaches in Fig. \ref{fig:laplacian_st_vs_eq}(a). The sequential approach shows a linear growth of the computing time, as expected, since it is solving $K$ times problems in a sequential fashion. Since the current implementation of space-time preconditioners in FEMPAR does not exploit local sequentiality in time, we observe a discrepancy between the intersection of curves in terms of local solves and {elapsed} time, due to the quadratic complexity of sparse direct methods. In any case, the space-time approach starts to be competitive with less than 10 time partitions. Out of these plots, we are able to reduce the time-to-solution of simulations with the space-time approach by adding more processors to exploit time-parallelism, both for the linear and nonliner problems considered herein. Besides, the method shows excellent weak scalability properties.

\subsubsection{Space-time CDR solver} Consider the CDR equation \Eq{cd} with $\sigma=10^{-4}$, $\beta= (1, 0)$, and $f=1$, on an original domain $\Omega \times (0,T] =[0,0.3]^2 \times (0,0.3]$, and scaled through the weak scalability analysis by $P_x, P_t$ respectively. Homogeneous Dirichlet and initial conditions are enforced. Local problem size is fixed with $K_n=30$ and $\frac{H}{h}=30$. The ratio between spatial and time partition is $\frac{P_x}{P_t}=3$. Several diffusion parameters are considered in order to present different scenarios: from a diffusive-dominated case ($\nu=1.0$) to a convection-dominated one $(\nu=10^{-6})$. We have the convective CFL equal to 1.0 in all cases. SUPG stabilization is again used.

 \begin{table}\centering
\begin{tabular}{|c|c|c|c|c|c|c|c|}
\hline
  \multicolumn{2}{|r}{$\nu=$}                      & $1$  &   $10^{-1}$  &    $10^{-2}$ &$10^{-3}$ &$10^{-4}$&$10^{-6}$\\ \hline 
   \multicolumn{2}{|r}{$\text{CFL}_\nu$=}               & $10^2$ &$10$ &$1$ & $10^{-1}$ & $10^{-2}$ & $10^{-4}$\\ \hline 
   \multicolumn{2}{|r}{P\'eclet$=$}          & $5\cdot10^{-3}$ &$5\cdot10^{-2}$ &$0.5$ & 5 & 50 & $5\cdot10^{3}$  \\ \hline 
 $(P_x \times P_y)\times P_t$ & \#Sbd  & \multicolumn{6}{c|}{}              \\ \hline
 $(3\times3)\times 1$    & 9     &   18   &    11     &    7     &    5    &   5    &    5      \\ \hline
 $(6\times6)\times 2$    & 72    &   28   &    16     &   11     &   11    &   11   &   11      \\ \hline          
 $(9\times9)\times 3$    & 243   &   35   &    17     &   11     &   12    &   12   &   13      \\ \hline 
 $(12\times12)\times 4$  & 576   &   37   &    17     &   12     &   13    &   14   &   15      \\ \hline    
 $(15\times15)\times 5$  & 1125  &   38   &    17     &   13     &   14    &   15   &   17      \\ \hline 
 $(18\times18)\times 6$  & 1944  &   39   &    17     &   13     &   15    &   15   &   18      \\ \hline 
 $(21\times21)\times 7$  & 3087  &   40   &    17     &   14     &   16    &   16   &   19      \\ \hline    
 $(24\times24)\times 8$  & 4608  &   41   &    18     &   14     &   17    &   17   &   21      \\ \hline 
 $(27\times27)\times 9$  & 6561  &   41   &    18     &   15     &   18    &   18   &   22      \\ \hline
 $(30\times30)\times 10$ & 9000  &   42   &    18     &   16     &   19    &   19   &   24        \\ 
 \hline
\end{tabular}
\caption{Iteration count for CDR equation. Convective CFL equal to 1.0 in all cases. CFL$_\nu$ represents the diffusive CFL.}
\label{tab-cdr_stbddc_iters}
\end{table}

Table \ref{tab-cdr_stbddc_iters} presents the iteration count for different diffusion values that lead to different scenarios. In the diffusive-dominated case the STBDDC preconditioned GMRES tends to an asymptotically constant number of iterations, thus independent of the number of subdomains. Moving to the convective case, the number of iterations slightly grows with the decrease of the diffusive CFL number.

\subsection{Nonlinear space-time $p$-Laplacian solver} In this experiment we compare the sequential-in-time method (solving the spatial problem with a BDDC approach for every time step) against the proposed STBDDC solver. We consider the $p$-Laplacian problem (with $p=1$), i.e., \Eq{cd} with $\nu = | \nabla u|$, $\beta ( 0,0)$, and $\sigma = 0$, on $\Omega=[0,1]^2 \times (0, T]$ with the initial solution $u^0 = x+y$, Dirichlet data $g=x+y$, and the forcing term $f=1$. We fix the time step size to $ | \elt | =10^{-3} $ and $\frac{H}{h}=30$. 

We consider the same setting as the experiment reported in Fig. \ref{fig:laplacian_st_vs_eq} to compare the sequential-in-time and space-time parallel solvers for a nonlinear problem. Thus, we consider an increasing number of time steps, and  $P_{s} = P_x \times P_y = 4^2$ processors for the space-time solver. The sequential solver exploits $P_s$ processors to extract space-parallelism only.
\begin{figure}[h!]
\begin{center}
\begin{tabular}{@{}c@{}c@{}}
\includegraphics[width=0.47\textwidth]{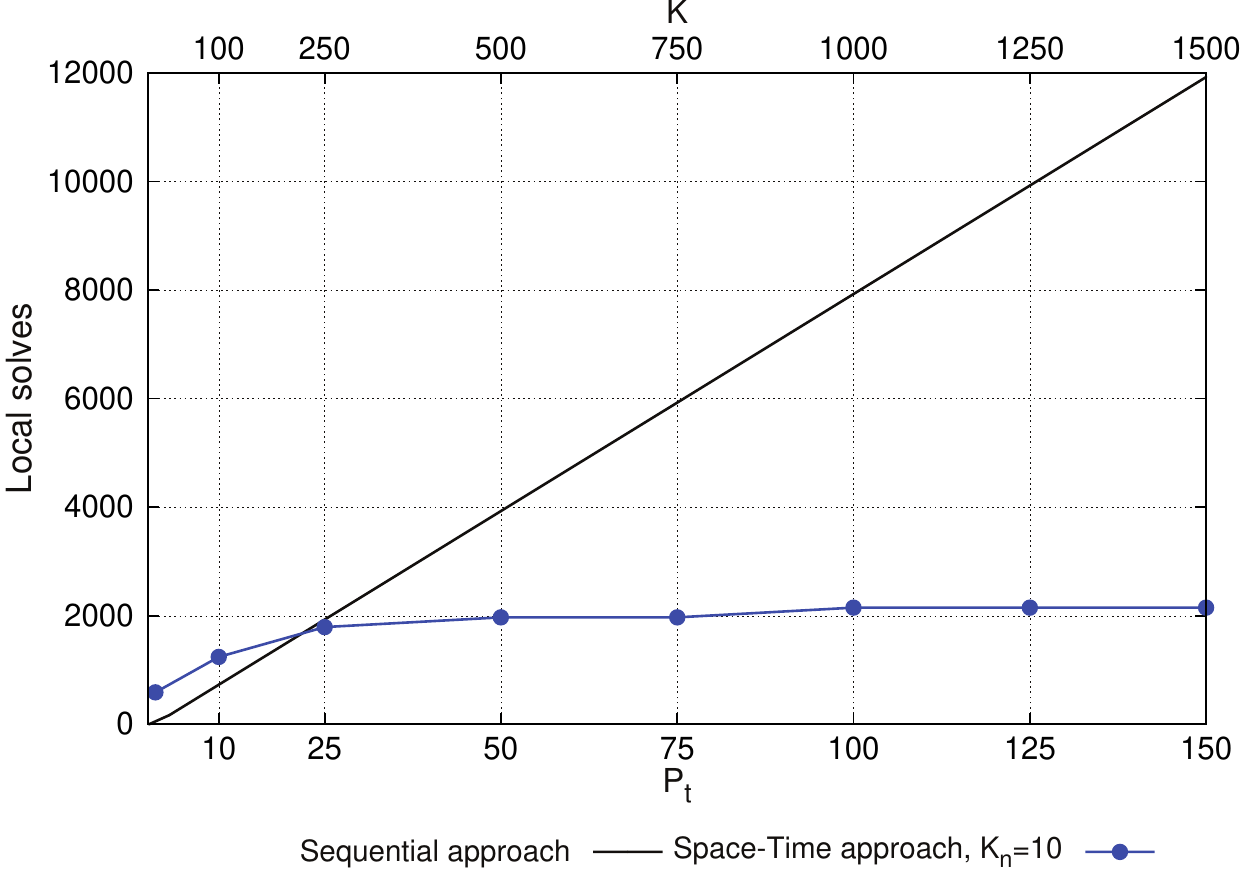} &
\includegraphics[width=0.47\textwidth]{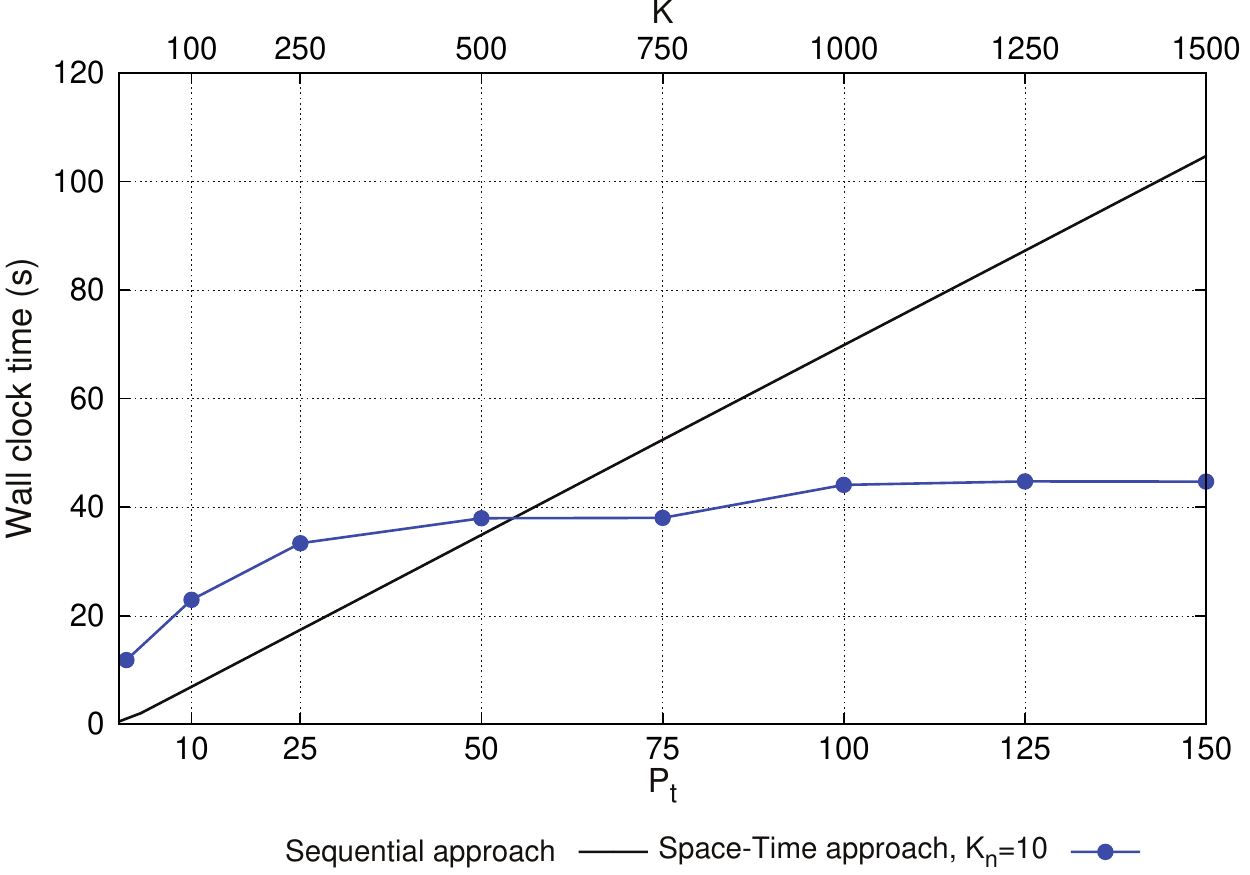} \\
\end{tabular}
\end{center}
\caption{ Comparison between the sequential and space-time solvers for the $p$-Laplacian transient problem on MareNostrum supercomputer. Spatial partition is fixed to $P_x=P_y=4$. The space-time approach is using $P=4^2 P_t$ MPI fine-level tasks. }
\label{fig:p_laplacian}
\end{figure}
Figure \ref{fig:p_laplacian} shows the comparison between the sequential approach up to $T=1.5$ and $K=1500$ time steps.  In Fig. \ref{fig:p_laplacian}(a) we plot the number of space solves, which is now proportional to the number of accumulated linear iterations through the nonlinear iterations. Remarkable algorithmic scalability is also obtained in the nonlinear setting. {Elapsed} time plots in \ref{fig:p_laplacian}(b) show a similar behaviour in the nonlinear case as in the linear one. The nonlinear case also exhibits excellent weak scalability properties. The nonlinear space-time solver is competitive in terms of number of local problems to be computed at about 20 processors, whereas  it requires slightly more than 50 processors in order to be superior in terms of {elapsed} time. Discrepancies should be substantially reduced exploiting causality for the local problems, as commented above. Out of these plots, we are able to reduce the time-to-solution of simulations with the space-time approach by adding more processors to exploit time-parallelism for the nonliner problem considered herein. In any case, there is still room for improvement when considering nonlinear problems. In this sense, nonlinear space-time preconditioning \cite{klawonn_nonlinear_2014,cai_nonlinearly_2002} and more elaborated linearization strategies have the potential to lead to better performance.
 
\section{Conclusions}\label{sec:conclusions}

In this work, we have considered a space-time iterative solver based on DD techniques. In particular, we have considered the GMRES iterative solver with space-time preconditioning obtained by extending the BDDC framework to space-time for parabolic problems discretized with FEs. Since the time direction has a very different nature than the spatial one, i.e., it is a transport-type operator, a particular definition of the coarse DOFs and transfer operators is considered, taking into account time causality. Further, perturbation terms must be included to lead to a well-posed system. The exposition has been carried out for a Backward-Euler time integrator, but the extension to $\theta$-methods and Runge-Kutta methods is straightforward. Further, the well-posedness of the proposed space-time preconditioner has been checked.

On the other hand, we have carried out a detailed set of numerical experiments on parallel platforms. Out of these results, the proposed methodology is observed to exhibit excellent scalability properties. The methods are weakly scalable in time, i.e., increasing X times the number of MPI tasks one can solve X times more time steps, in approximately the same amount of time, which is a key property to reduce time-to-solution in transient simulations with heavy time stepping. We have also shown weak scalability in space-time, where one is not only facing larger problems in time but also in space. Further, we have applied the STBDDC preconditioner to nonlinear problems, by considering a linearization of the full space-time system, and applying the proposed space-time solver at every nonlinear iteration. The use of the space-time solvers proposed becomes faster than a sequential-in-time approach for a modest number of processors.

Future work will include the development of nonlinear space-time BDDC preconditioners, extending the concept of nonlinear preconditioning (see, e.g., the recent work in \cite{klawonn_nonlinear_2014} for nonlinear FETI preconditioners in space) to space-time. As it has already been observed in space \cite{cai_nonlinearly_2002}, the use of nonlinear preconditioning should make space-time preconditioning more effective for nonlinear problems. Further extensions of this work will involve the extension to multilevel space-time algorithms (to keep perfect weak scalability at larger scales), and their application to solid mechanics and incompressible fluid dynamics simulations.

\section*{Acknowledgments}

\bibliographystyle{plain}
\bibliography{refs}

\begin{thebibliography}{10}

\bibitem{badia_implementation_2013}
S.~Badia, A.~F. Mart\'in, and J.~Principe.
\newblock Implementation and scalability analysis of balancing domain
  decomposition methods.
\newblock {\em Archives of Computational Methods in Engineering},
  20(3):239--262, 2013.

\bibitem{badia_highly_2014}
S.~Badia, A.~F. Mart\'in, and J.~Principe.
\newblock A highly scalable parallel implementation of balancing domain
  decomposition by constraints.
\newblock {\em SIAM Journal on Scientific Computing}, 36(2), 2014.

\bibitem{badia_multilevel_2016}
S.~Badia, A.~F. Mart\'in, and J.~Principe.
\newblock Multilevel balancing domain decomposition at extreme scales.
\newblock {\em SIAM Journal on Scientific Computing}, pages C22--C52, 2016.

\bibitem{badia_scalability_2015}
Santiago Badia, Alberto~F. Mart\'in, and Javier Principe.
\newblock On the scalability of inexact balancing domain decomposition by
  constraints with overlapped coarse/fine corrections.
\newblock {\em Parallel Computing}, 50:1--24, 2015.

\bibitem{Brenner2010}
Susanne~C. Brenner and Ridgway Scott.
\newblock {\em The Mathematical Theory of Finite Element Methods}.
\newblock Springer, 3rd edition edition, 2010.

\bibitem{brooks_streamline_1982}
Alexander~N. Brooks and Thomas~J.R. Hughes.
\newblock Streamline upwind/{Petrov}-{Galerkin} formulations for convection
  dominated flows with particular emphasis on the incompressible
  {Navier}-{Stokes} equations.
\newblock {\em Computer Methods in Applied Mechanics and Engineering},
  32(1–3):199--259, September 1982.

\bibitem{cai_nonlinearly_2002}
Xiao-Chuan Cai and David~E. Keyes.
\newblock Nonlinearly preconditioned inexact {Newton} algorithms.
\newblock {\em SIAM Journal on Scientific Computing}, 24(1):183--200, 2002.

\bibitem{cyr_stabilization_2012}
Eric~C. Cyr, John~N. Shadid, and Raymond~S. Tuminaro.
\newblock Stabilization and scalable block preconditioning for the
  {Navier}--{Stokes} equations.
\newblock {\em Journal of Computational Physics}, 231(2):345--363, 2012.

\bibitem{dohrmann_preconditioner_2003}
Clark~R. Dohrmann.
\newblock A preconditioner for substructuring based on constrained energy
  minimization.
\newblock {\em SIAM Journal on Scientific Computing}, 25(1):246--258, 2003.

\bibitem{dryja_bddc_2007}
Maksymilian Dryja, Juan Galvis, and Marcus Sarkis.
\newblock {BDDC} methods for discontinuous {Galerkin} discretization of
  elliptic problems.
\newblock {\em Journal of Complexity}, 23(4–6):715--739, 2007.

\bibitem{pfasst}
Matthew Emmett and Michael~J. Minion.
\newblock Toward an efficient parallel in time method for partial differential
  equations.
\newblock {\em Comm. App. Math. and Comp. Sci.}, 7(1):105--132, 2012.

\bibitem{falgout_parallel_2014}
R.~Falgout, S.~Friedhoff, T.~Kolev, S.~MacLachlan, and J.~Schroder.
\newblock Parallel time integration with multigrid.
\newblock {\em SIAM Journal on Scientific Computing}, 36(6):C635--C661, 2014.

\bibitem{gander_parareal_2013}
Martin~J. Gander, Rong-Jian Li, and Yao-Lin Jiang.
\newblock Parareal {Schwarz} waveform relaxation methods.
\newblock In {\em Domain {Decomposition} {Methods} in {Science} and
  {Engineering} {XX}}, volume Part II of {\em Lecture {Notes} in
  {Computational} {Science} and {Engineering}}, pages 451--458. Springer Berlin
  Heidelberg, 2013.

\bibitem{gander_analysis_2014}
Martin~J. Gander and Martin Neumüller.
\newblock Analysis of a {New} {Space}-{Time} {Parallel} {Multigrid} {Algorithm}
  for {Parabolic} {Problems}.
\newblock {\em arXiv:1411.0519 [math]}, 2014.
\newblock arXiv: 1411.0519.

\bibitem{klawonn_nonlinear_2014}
A.~Klawonn, M.~Lanser, and O.~Rheinbach.
\newblock Nonlinear {FETI}-{DP} and {BDDC} {Methods}.
\newblock {\em SIAM Journal on Scientific Computing}, 36(2):A737--A765, 2014.

\bibitem{parareal}
J.L. Lions, Y.~Maday, and A.~Turinici.
\newblock A parareal in time discretization of {PDE}s.
\newblock {\em Acad. Sci. Paris}, 332(Seria I):661--668, 2001.

\bibitem{Mandel2003}
Jan Mandel and Clark~R. Dohrmann.
\newblock Convergence of a balancing domain decomposition by constraints and
  energy minimization.
\newblock {\em Numerical Linear Algebra with Applications}, 10(7):639--659,
  2003.

\bibitem{mandel_multispace_2008}
Jan Mandel, Bedřich Soused\'ik, and Clark Dohrmann.
\newblock Multispace and multilevel {BDDC}.
\newblock {\em Computing}, 83(2):55--85, 2008.

\bibitem{minion_interweaving_2015}
M.~L. Minion, R.~Speck, M.~Bolten, M.~Emmett, and D.~Ruprecht.
\newblock Interweaving {PFASST} and {Parallel} {Multigrid}.
\newblock {\em SIAM Journal on Scientific Computing}, 37(5):S244--S263, January
  2015.

\bibitem{speck_space-time_2013}
Robert Speck, Daniel Ruprecht, Matthew Emmett, Matthias Bolten, and Rolf
  Krause.
\newblock A space-time parallel solver for the three-dimensional heat equation.
\newblock {\em arXiv:1307.7867 [cs, math]}, 2013.

\bibitem{Toselli2005}
A.~Toselli and O.~Widlund.
\newblock {\em Domain Decomposition Methods - Algorithms and Theory}.
\newblock Springer-Verlag, 2005.

\bibitem{tu_three-level_2007}
Xuemin Tu.
\newblock Three-level {BDDC} in three dimensions.
\newblock {\em {SIAM} Journal on Scientific Computing}, 29(4):1759--1780, 2007.

\bibitem{weinzierl_geometric_2012}
Tobias Weinzierl and Tobias K\"oppl.
\newblock A geometric space-time multigrid algorithm for the heat equation.
\newblock {\em Numerical Mathematics: Theory, Methods and Applications},
  5(01):110--130, 2012.

\end{thebibliography}

\end{document}